\documentclass[a4paper,10pt]{amsart}
\usepackage{amssymb}
\usepackage{amsmath}
\usepackage[T1]{fontenc}
\usepackage{tipa}
\usepackage{mathrsfs}
\usepackage{amsfonts}
\usepackage{amsthm}
\usepackage{enumerate}
\usepackage{color}
\usepackage{hyperref}
\usepackage[leqno]{amsmath}
\usepackage{tikz}

\newtheorem{thm}{Theorem}
\newtheorem{defn}{Definition}
\newtheorem{fact}{Fact}
\newtheorem{lem}[thm]{Lemma}
\newtheorem{prop}[thm]{Proposition}
\newtheorem{cor}[thm]{Corollary}
\newtheorem{question}{Question}
\newtheorem*{kthm}{Kronecker's Theorem}

\newcommand{\up}{\upharpoonright}
\newcommand{\ra}{\rightarrow}
\newcommand{\mc}{\mathcal}
\newcommand{\ms}{\mathscr}
\makeatletter
\newcommand{\leqnomode}{\tagsleft@true}
\newcommand{\reqnomode}{\tagsleft@false}
\makeatother

\begin{document}

  \title{L vector spaces and L fields}
  \author{Yinhe Peng}
  \address{Academy of Mathematics and Systems Science, Chinese Academy of Sciences\\ East Zhong Guan Cun Road No. 55\\Beijing 100190\\China}
\email{pengyinhe@amss.ac.cn}

  \author{Liuzhen Wu}
  \address{Academy of Mathematics and Systems Science, Chinese Academy of Sciences\\ East Zhong Guan Cun Road No. 55\\Beijing 100190\\China}
\email{lzwu@math.ac.cn}

\thanks{Peng was partially supported by NSFC No. 11901562 and a program of the Chinese Academy of Sciences. Wu was partially supported by NSFC No. 11871464.}

  \maketitle
  \begin{abstract}
  We construct in ZFC an L topological vector space---a topological vector space that is an L space---and an L field---a topological field that is an L space. This generalizes  results in \cite{moore06} and \cite{pw}.

  \end{abstract}
  \section{Introduction}
  One   important and interesting problem in set-theoretic topology is the topological basis problem.
  \begin{question}\label{basisp}
  Which class of topological spaces has a 3-element basis?
  \end{question}
  
  Moore constructed in \cite{moore06} an \emph{L space}---a hereditarily Lindel\"of regular space that is not separable---and showed that Question \ref{basisp} has a negative answer in the class of regular spaces. We constructed in \cite{pw} an \emph{L group}---a topological group that is an L space---and showed that Question \ref{basisp} has a negative answer in the class of topological groups.
  
 It is natural to expect that some strengthened   algebraic property may lead to a positive answer to Question \ref{basisp}. E.g., does Question \ref{basisp} have a positive answer in the class of topological fields or topological vector spaces? 
 
 One example that stronger algebraic property makes differences is the metrization of pseudo-compact spaces.  Shakhmatov proved that every pesudo-compact subspace of a topological field is metrizable. But the situation for topological groups is different since every completely regular space can be embedded into a topological group---$\mathbb{R}^\kappa$.\medskip

  The importance of stronger algebraic properties has been realized quite early. For example,   the following  are Question 3.7 and Question 3.8 in \cite{shak}.
    \begin{question}[\cite{shak}]\label{q1}
    Let $\mc{P}$ be one of the following properties. Is there a topological vector space $L$ with property $\mc{P}$ such that $L\times L$ does not have property $\mc{P}$? Or is there such a topological field?
  \begin{enumerate}[(i)]
  \item normality;
  \item paracompactness;
  \item Lindel\"ofness.
  \end{enumerate}
  \end{question}

  \begin{question}[\cite{shak}]\label{q2}
  \begin{enumerate}[(i)]
  \item Is there a Lindel\"of topological vector space $L$ such that $L^2$ is not normal?
  \item Is there a Lindel\"of topological field $P$ such that $P^2$ is not normal?
  \end{enumerate}
  \end{question}
  
  We answer Question \ref{q1} and Question \ref{q2} positively by constructing  examples in ZFC   that are in addition  L spaces. These examples also answer Question \ref{basisp} negatively in the class of topological vector spaces and in the class of topological fields.
  
  \begin{thm}\label{thm1}
  There is an L vector space whose square is neither normal nor weakly paracompact.
  \end{thm}
  An \emph{L vector space} is a topological vector space that is also an L space.  Throughout the paper, all topological vector spaces are over the same scalar field $\mathbb{R}$.
  
  \begin{thm}\label{thm2}
  There is an L field whose square is neither normal nor weakly paracompact.
  \end{thm}
  An \emph{L field} is a topological field that is also an L space. The L field in Theorem \ref{thm2}   has characteristic 0.
  
  We will construct examples for Theorem \ref{thm1} and Theorem \ref{thm2} in Section 3 and Section 4 respectively. The construction is based on Todorcevic's powerful method of \emph{minimal walks} in \cite{st87}, Moore's ingenious construction of a new combinatorial function---$osc$---in \cite{moore06} and more combinatorial properties of $osc$ found in our earlier investigation in \cite{pw}. This construction and earlier results, e.g., \cite{moore06}, \cite{moore08}, \cite{p}, \cite{pw}, suggest that Moore's idea and more importantly, Todorcevic's technique of minimal walk may have potential to reveal more phenomenon.  Further applications probably rely  on discoveries of further combinatorial properties of, e.g., $osc$. So in Section 5, we investigate further combinatorial properties of $osc$  and prove in Theorem \ref{**consistency} and Theorem \ref{t11} that a combinatorial property $(**)$ is independent of \rm{ZFC}. This indicates a limitation of \rm{ZFC} construction using $osc$.  
  
  The paper is organized as follows.  In Section 3, we construct an L vector space and prove Theorem \ref{thm1}. In Section 4, we construct an L field and prove Theorem \ref{thm2}. In Section 5, we investigate further combinatorial properties of the oscillation map $osc$. Readers who are interested in corresponding topic can go to corresponding section directly.

  \section{preliminary}
In this section, we introduce some standard definitions and facts.
We start with the following notations associated to Todorcevic's minimal walks. Most originate in \cite{st87} and \cite{st07} except for $L$ which originates in \cite{moore06}. For a modern and encyclopedic exposition of minimal walks, see \cite{st07}. For the following sections, we assume that readers are   familiar with the combinatorial properties of $osc$ in \cite{moore06} and \cite{pw}.
\begin{defn}
  \begin{enumerate}
    \item $[X]^\kappa$ is the set of all subsets of $X$ of size $\kappa$. Moreover, if $X$ is a set of ordinals, $k<\omega$ and $b\in [X]^k$, then $b(0),b(1),...,b(k-1)$ is the increasing enumeration of $b$ and
    throughout this paper, we will use $(b(0),...,b(k-1))$ to denote $b$.
    \item Suppose that $a$, $b$ are both finite sets of ordinals but neither is an ordinal. Say $a<b$ if $\max a<\min b$. For ordinals $\alpha,\beta$, $\alpha<b$ denotes $\{\alpha\}< b$, $a<\beta$ denotes $a<\{\beta\}$ and $\alpha\leq b$ denotes $\alpha\leq \min b$.
  \end{enumerate}
\end{defn}

\begin{defn}

 A \emph{C-sequence} is a sequence $\langle C_\alpha : \alpha< \omega_1\rangle$ such that
  $C_{\alpha+1}=\{\alpha\}$ and $C_\alpha$ is a cofinal
subset of $\alpha$ of order type $\omega$ for limit $\alpha$'s.
\end{defn}

Roughly speaking, the \emph{minimal walk} from $\beta$ towards a smaller ordinal $\alpha$ is the sequence $\beta=\beta_0>\beta_1>...>\beta_n=\alpha$ such that for each $i<n$,
$\beta_{i+1}=\min C_{\beta_i}\setminus \alpha$.  Here the \emph{weight} of the step from $\beta_i$ to $\alpha$ is $|C_{\beta_i}\cap \alpha|$.

\begin{defn}[\cite{st87}]
For a $C$-sequence, the \emph{maximal weight} of the walk is the function $\varrho_1 : [\omega_1]^2\rightarrow \omega$,
defined recursively by $\varrho_1(\alpha,\beta)=\max \{|C_\beta\cap \alpha|, \varrho_1(\alpha, \min (C_\beta\setminus \alpha))\}$
with boundary value $\varrho_1(\alpha,\alpha)=0$. $\varrho_{1\beta}: \beta\rightarrow \omega$ is defined by $\varrho_{1\beta}(\alpha)=\varrho_1(\alpha,\beta)$ for $\alpha<\beta$.
\end{defn}

Intuitively, the function $\varrho_1$ is the function constructed from the C-sequence which records the maximal weight of all steps in one walk.

We also need the following splitting function.
\begin{defn}
For $\alpha<\beta<\omega_1$, $\Delta(\alpha,\beta)=\min (\{\xi<\alpha: \varrho_1(\xi,\alpha)\neq \varrho_1(\xi,\beta)\}\cup \{\alpha\})$.
\end{defn}

The following two trace functions will be needed.
\begin{defn}[\cite{st07}]
For a given C-sequence, the \emph{upper trace} $Tr: [\omega_1]^2\rightarrow [\omega_1]^{<\omega}$
is recursively defined for $\alpha\leq \beta<\omega_1$ as follows:
\begin{itemize}
  \item $Tr(\alpha,\alpha)=\{\alpha\}$;
  \item $Tr(\alpha,\beta)=(Tr(\alpha, \min (C_\beta\setminus \alpha))\cup \{\beta\})$.
\end{itemize}
\end{defn}

\begin{defn}[\cite{moore06}]
For a given C-sequence, the \emph{lower trace} $L: [\omega_1]^2\rightarrow [\omega_1]^{<\omega}$
is recursively defined for $\alpha\leq \beta<\omega_1$ as follows:
\begin{itemize}
  \item $L(\alpha,\alpha)=\emptyset$;
  \item $L(\alpha,\beta)=(L(\alpha, \min (C_\beta\setminus \alpha))\cup \{\max (C_\beta\cap \alpha)\})\setminus \max (C_\beta\cap \alpha)$.
\end{itemize}
\end{defn}

We recall the following properties of these functions:
\begin{defn}
  A function $a: [\omega_1]^2\rightarrow \omega$ (or a sequence $\langle a_\beta:\beta<\omega_1\rangle$ where $a_\beta: \beta\rightarrow \omega$) is coherent if for any $\alpha<\beta<\omega_1$,
  $\{\xi<\alpha: a(\xi,\alpha)\neq a(\xi,\beta)\}$ is finite.
\end{defn}

\begin{fact}[\cite{st87}]\label{f0}
$\varrho_1$ is coherent and finite-to-one.
\end{fact}
\begin{fact}[\cite{moore06}]\label{f1}
(1) For limit ordinal $\beta>0$, $\underset{\alpha\rightarrow \beta}{lim}\ \min L(\alpha,\beta)=\beta$.

    (2)  For $\alpha<\beta<\gamma$, if $L(\beta,\gamma)<L(\alpha,\beta)$, then $L(\alpha,\gamma)=L(\alpha,\beta)\cup L(\beta,\gamma)$. 
  \end{fact}
  
  The oscillation of two finite functions is well-known. Suppose that $s$ and $t$ are two functions defined on a common
finite set of ordinals $F$. $Osc(s, t; F)$ is the set of all $\xi$ in $F \setminus \{\min F\}$ such that
$s(\xi^-) \leq t(\xi^-)$ and $s(\xi) > t(\xi)$ where $\xi^-$ is the greatest element of $F$ less than $\xi$.

The $osc$ map is then induced from the functions $\varrho_1$ and $L$.
\begin{defn}[\cite{moore06}]
For $\alpha<\beta<\omega_1$, $Osc(\alpha,\beta)$ denotes $Osc(\varrho_{1\alpha},\varrho_{1\beta}; L(\alpha,\beta))$ and
$osc(\alpha,\beta)=|Osc(\alpha,\beta)|$ denotes the cardinality of $Osc(\alpha,\beta)$.
\end{defn}



  We will also need Kronecker's Theorem.\footnote{A proof is available in \cite{jwsc}.}
\begin{kthm}[\cite{lk}]
  Let $A$ be a real $m\times n$ matrix and assume that $\{\textbf{z}\in \mathbb{Q}^m: A^T \textbf{z}\in \mathbb{Q}^n\}=\{\textbf{0}\}$.
  Then for any $\epsilon>0$, for any $b_0,...,b_{m-1}\in\mathbb{R}$, there exist $p_0,...,p_{m-1}\in \mathbb{Z}$, $\textbf{q}\in \mathbb{Z}^n$ such that
  $|A_i \textbf{q}-p_i-b_i|<\epsilon$ for all $i<m$ where $A_i$ is the $i$th row of $A$.
\end{kthm}
 We will also need the following combinatorial property \cite[Theorem 4.3]{moore06} of $osc$ and its generalization. 
 \begin{thm}[\cite{moore06}]
 For every $\mathscr{A}\subset [\omega_1]^k$ and $\mathscr{B}\subset [\omega_1]^l$ which are uncountable families
of pairwise disjoint sets and every $n < \omega$, there are $a$ in $\mathscr{A}$ and $b_m$ $(m<n)$ in $\mathscr{B}$
such that for all $ i < k, j < l$, and $m < n$:
$a < b_m$ and

$osc(a(i),b_m(j))=osc(a(i),b_0(j))+m$.
 \end{thm}

 The arguments in \cite{moore06} can easily be adapted to show  that in the proof of    \cite[Theorem 4.3]{moore06}, for any possible candidates of $b_m (m<n)$, there are unbounded many $a\in \mathscr{A}\cap M$  satisfying the requirement of the Theorem. So we have the following (see also    \cite[Lemma 6]{pw}).
\begin{lem}[\cite{moore06}]\label{moore}
For every $\mathscr{A}\subset [\omega_1]^k$ and $\mathscr{B}\subset [\omega_1]^l$ which are uncountable families
of pairwise disjoint sets and every $n < \omega$, there are $\{a_p: p<\omega\}\subset \mathscr{A}$ and $\{b_m: m < n\}\subset\mathscr{B}$
such that for all $p<\omega, i < k, j < l$, and $m < n$:
$a_p < b_m$ and

$osc(a_p(i),b_m(j))=osc(a_p(i),b_0(j))+m$.
\end{lem}

\section{Topological vector space}

A \emph{topological vector space} is a vector space endowed with a topology with respect to which the vector operations are continuous. We only consider vector spaces over $\mathbb{R}$.

 In this section we will   construct an \emph{L vector space}---a topological vector space that is an L space---whose square is neither normal nor weakly paracompact. A collection $\mathscr{U}$ of subsets of a space $X$ is \emph{point finite} if every $x\in X$ is in at most finitely many sets in $\mathscr{U}$. A space $X$ is \emph{weakly paracompact} or \emph{metacompact} if every open cover has a point finite open refinement.\medskip
 
   Let us first fix some notations.
  \begin{defn}
\begin{enumerate}
  \item For a real number $x$, $[x]$ is the greatest integer less than or equal to $x$. $frac(x)=x-[x]$ and $frac(X)=\{frac(x): x\in X\}$ for $X\subset \mathbb{R}$.


  \item For $X\subset \mathbb{R}^n$ and $x\in \mathbb{R}^n$, $x$ is a \emph{complete accumulation point} of $X$ if for any open neighbourhood $O$ of $x$, $|O\cap X|=|X|$.
  \end{enumerate}
\end{defn}
Now   fix a rationally independent set of reals $\{\theta_\alpha:\alpha<\omega_1\}$.
  \begin{defn}
  Throughout this section, denote $f(x)=\frac{1}{x}\sin\frac{1}{x}$ for $x\in \mathbb{R}\setminus\{0\}$.
  \begin{enumerate}
\item $\mathscr{L}=\{w_\beta: \beta<\omega_1\}$ where
\begin{displaymath}
   w_\beta(\alpha) = \left\{
     \begin{array}{lr}
       f(frac(\theta_\alpha osc(\alpha,\beta)+\theta_\beta))  & : \alpha<\beta\\
       \theta_\beta & : \alpha=\beta\\
       0 & : \alpha>\beta .
     \end{array}
   \right.
\end{displaymath}
We view $\mathscr{L}$ as a subspace of $\mathbb{R}^{\omega_1}$, equipped with the product topology.
\item $vec(\mathscr{L})$ is the topological vector subspace of $(\mathbb{R}^{\omega_1},+, \cdot)$ generated by $\mathscr{L}$, where the vector addition $+$ and scalar multiplication $\cdot$ are coordinatewise addition and coordinatewise multiplication.
\end{enumerate}
  \end{defn} 
   We will show that $vec(\mathscr{L})$ is an L vector space. We first prove the following property which is stronger than  hereditary Lindel\"ofness.
  
  \begin{prop}\label{prop1}
  Let $\mathscr{A}\subset [\omega_1]^k$ and
    $\mathscr{B}\subset [\omega_1]^l$ be uncountable families of pairwise disjoint sets.  Then for every sequence of non-empty open sets $\langle U_i\subset\mathbb{R}: i<k \rangle$ and every assignment $\{\langle r_{b(j)}\in \mathbb{R}\setminus \{0\}: j<l\rangle: b\in \mathscr{B}\}$, there are $a\in \mathscr{A}$ and $ b\in \mathscr{B}$ such that $a<b$ and  for all $i<k$, $\underset{j<l}{\sum} r_{b(j)}w_{b(j)}(a(i))\in U_i$.
\end{prop}
\begin{proof}
  Going to uncountable subfamilies we may assume, by \cite[Theorem 7]{pw} (see also Theorem \ref{pw} in Section 5), that there is a sequence $\langle c_{ij}: i<k,j<l\rangle \in \mathbb{Z}^{k\times l}$ such that for any $a\in \mathscr{A}$, for any
    $b\in \mathscr{B}$, if $a<b$, then $osc(a(i),b(j))=osc(a(i),b(0))+c_{ij}$ whenever $i<k,j<l$. Note that $c_{i0}=0$ for any $i<k$.\smallskip


Fix  a countable elementary submodel $M\prec H(\aleph_2)$ containing everything relevant.\footnote{See \cite[Section 3]{moore06} for basic facts about elementary submodels.} Fix for now $a\in \ms{A}\setminus M$ and $b\in \ms{B}\setminus M$.

    For each $i<k$, define $F_i(x)=\underset{j<l}{\sum} r_{b(j)}f(frac(x+\theta_{a(i)}c_{ij}+\theta_{b(j)}))$. Note that $\theta_{a(i)}$'s and $\theta_{b(j)}$'s are rationally independent. So $F_i(x)$ is defined and continuous on some small interval containing $-\theta_{b(0)}$. Hence
    \begin{eqnarray*}
    && \lim_{x\rightarrow -\theta_{b(0)}} \sum_{0<j<l} r_{b(j)}f(frac(x+\theta_{a(i)}c_{ij}+\theta_{b(j)}))\\
    &=& \sum_{0<j<l} r_{b(j)}f(frac(-\theta_{b(0)}+\theta_{a(i)}c_{ij}+\theta_{b(j)})).
    \end{eqnarray*}
    Moreover, the range of $f(frac(x+\theta_{b(0)}))$ is $\mathbb{R}$ on $(-\theta_{b(0)},-\theta_{b(0)}+\epsilon)$ for any $\epsilon>0$. Choose $\epsilon$ small enough such that $F_i$ is continuous on $(-\theta_{b(0)},-\theta_{b(0)}+\epsilon)$. Then the range of $F_i$ on $(-\theta_{b(0)},-\theta_{b(0)}+\epsilon)$ is $\mathbb{R}$. Now choose, for each $j<l$, open sets $O_i,O_i',O_{ij}'', O_{ij}'''$ in $M$  such that
    \begin{enumerate}
    \item $O_i\subset (0,1)$, $\theta_{a(i)}\in O_i'$, $\theta_{b(j)}\in O_{ij}''$ and $r_{b(j)}\in O_{ij}'''$;

     \item $\underset{j<l}{\sum} r_jf(frac(x+yc_{ij}+z_j))\in U_i$ for all $\langle r_j\in O_{ij}''': j<l\rangle$, $x\in O_i$, $y\in O_i'$ and $\langle z_j\in O_{ij}'': j<l\rangle$.
    \end{enumerate}
     For $j<l$, let $O_j''=\underset{i<k}{\bigcap} O_{ij}''$ and $O_j'''=\underset{i<k}{\bigcap} O_{ij}'''$. Observe that $\underset{i<k}{\prod} O_i'$,  $\underset{j<l}{\prod} O_j''$ and $\underset{j<l}{\prod} O_j'''$ are open neighbourhoods of $(\theta_{a(0)},...,\theta_{a(k-1)})$, $(\theta_{b(0)},...,\theta_{b(l-1)})$ and $(r_{b(0)},...,r_{b(l-1)})$  respectively.
    Now
    $$\ms{A}'=\{a'\in \ms{A}: (\theta_{a'(0)},...,\theta_{a'(k-1)})\in  \underset{i<k}{\prod} O_i'\},$$
    $$\ms{B}'=\{b'\in \ms{B}: (\theta_{b'(0)},...,\theta_{b'(k-1)})\in  \underset{j<l}{\prod} O_j''\text{ and } (r_{b(0)},...,r_{b(l-1)}) \in \underset{j<l}{\prod} O_j'''\}$$
    are in $M$ since they are definable using parameters in $M$. By elementarity,  $\ms{A}'$ and $\ms{B}'$ 
    are uncountable since they are in $M$ but not subsets of $M$.\smallskip


    Finally for each $a\in\ms{A}'$ and $t\in [0,1]^k$, by Kronecker's Theorem, there is a natural number $m_{a,t}$ such that   $frac(\theta_{a(i)}m_{a,t}+t(i))\in O_i$ whenever $i<k$. By compactness of $[0,1]^k$, there is a natural number $m_a$ such that for each $t$ in $[0,1]^k$ and hence $\mathbb{R}^k$, there is some $m'<m_a$,   $frac(\theta_{a(i)}m'+t(i))\in O_i$ whenever $i<k$.
    Choose an uncountable subfamily $\mathscr{A}''\subset \mathscr{A}'$ and a natural number $m$ such that  $m_a=m$ for any $a\in\ms{A}''$. 
    
    Then for each $a\in \mathscr{A}''$ and
    $t\in \mathbb{R}^k$, there is a $m'<m$ such that   $frac(\theta_{a(i)}m'+t(i))\in O_i$ whenever $i<k$. Then by Lemma \ref{moore}, there are $a\in \mathscr{A}''$ and $\{b_s\in \mathscr{B}': s<m\}$ such
    that for all $i<k$ and $s<m$, $osc(a(i),b_s(0))=osc(a(i),b_0(0))+s$. Now, by our choice of $\mathscr{A}''$ and $m$, fix $s<m$ such that  
    $frac(\theta_{a(i)}osc(a(i),b_s(0)))\in O_i$ whenever $i<k$.\footnote{See also the proof of Theorem 5.6 in \cite{moore06}.} Together with (2) and our choice of $\ms{A}', \ms{B}'$, for all $i<k$,
    \begin{eqnarray*}
 && \sum_{j<l} r_{b_s(j)}w_{b_s(j)}(a(i))\\
  &=&\sum_{j<l} r_{b_s(j)}f(frac(\theta_{a(i)}osc(a(i),b_s(j))+\theta_{b_s(j)}))\\
  &=&\sum_{j<l}r_{b_s(j)}f(frac(\theta_{a(i)}osc(a(i),b_s(0))+\theta_{a(i)}c_{ij}+\theta_{b_s(j)}))\in U_i
 \end{eqnarray*}
 This finishes the proof of the proposition.
\end{proof}
Note that the proof of above proposition actually shows the following stronger property.
\begin{enumerate}
\item[($1\ms{B}$)] For $\ms{A},\ms{B}$, $\langle U_i: i<k\rangle$ and $\{r_\alpha\in \mathbb{R}\setminus\{0\}: \alpha<\omega_1\}$ as in Proposition \ref{prop1}, there are uncountable sub-families $\ms{A}''\subset\ms{A}$, $\ms{B}'\subset \ms{B}$ and $m<\omega$ such that whenever we are given $a\in \ms{A}''$ and $\{b_s\in \ms{B}': s<m\}$ satisfying   for all $i<k$ and $s<m$, $a<b_s$ and $osc(a(i),b_s(0))=osc(a(i),b_0(0))+s$, there is $s<m$ satisfying  for all $i<k$, $\sum_{j<l} r_{b_s(j)}w_{b_s(j)}(a(i))\in U_i$.
\end{enumerate}
Inductively applying ($1\ms{B}$), we get the following.
  \begin{cor}\label{cor13}
  Let $k\in \omega\setminus \{0\}$,  $n\in \omega\setminus \{0\}$,   $\langle l_p\in \omega\setminus\{0\}: p<n\rangle$ and $\mathscr{A}\subset [\omega_1]^k$,
    $\mathscr{B}_p\subset [\omega_1]^{l_p}$ ($p<n$) be uncountable families of pairwise disjoint sets. Then  for every sequence of non-empty open sets $\langle U_{i,p}\subset\mathbb{R}: i<k, p<n \rangle$ and   every assignment $\{ r_{\alpha}\in \mathbb{R}\setminus \{0\}: \alpha< \omega_1\}$, there are $a\in \mathscr{A}$ and $ \langle b_p\in \mathscr{B}_p: p<n \rangle$ such that for all $i<k$ and $p<n$, $a< b_p$ and 
    $\underset{j<l_p}{\sum} r_{b_p(j)}w_{b_p(j)}(a(i))\in U_{i,p}$.
\end{cor}
\begin{proof}
Inductively applying ($1\ms{B}$) $n$ times and enlarging $m$ if necessary, we may assume that for some uncountable $\ms{A}'\subset\ms{A}$, $\{\ms{B}_p'\subset \ms{B}_p: p<n\}$ and $m<\omega$,
\begin{enumerate}
\item whenever we are given $a\in \ms{A}'$, $p<n$ and $\{b_{s,p}\in \ms{B}'_p: s<m\}$ satisfying   for all $i<k$ and $s<m$,  $a<b_{s,p}$ and $osc(a(i),b_{s,p}(0))=osc(a(i),b_{0,p}(0))+s$, there is $s<m$ satisfying for all $i<k$, $\sum_{j<l} r_{b_{s,p}(j)}w_{b_{s,p}(j)}(a(i))\in U_{i,p}$.
\end{enumerate}

Now choose an uncountable pairwise disjoint $\ms{B}'\subset [\omega_1]^{\sum_{p<n} l_p}$ such that 
\begin{enumerate}\setcounter{enumi}{1}
\item for each $b\in \ms{B}'$ and   $p<n$, $\{b(j): \sum_{q<p} l_q\leq j<\sum_{q\leq p} l_q\}\in \ms{B}'_p$.
\end{enumerate}
 Then by Lemma \ref{moore}, there are $a\in \mathscr{A}'$ and $\{b_s\in \mathscr{B}': s<m\}$ such
    that 
    \begin{enumerate}\setcounter{enumi}{2}
    \item  for all $i<k,j<\sum_{p<n} l_p$ and $s<m$, $a< b_s$ and $osc(a(i),b_s(j))=osc(a(i),b_0(j))+s$.
    \end{enumerate}
    Now fix $p<n$. By (2), for all $s<m$, $b_{s,p}\stackrel{\triangle}{=} \{b_s(j): \sum_{q<p} l_q\leq j<\sum_{q\leq p} l_q\}\in \ms{B}'_p$. By (3),   for all $i<k$ and $s<m$, $osc(a(i),b_{s,p}(0))=osc(a(i),b_{0,p}(0))+s$. By (1),   there is $s_p<m$ such that  for all $i<k$,  $\sum_{j<l} r_{b_{s_p,p}(j)}w_{b_{s_p,p}(j)}(a(i))\in U_{i,p}$. Now $a$ and $\langle b_{s_p,p}: p<n\rangle$ are as desired.
\end{proof}

 We now prove that the topological vector space $vec(\ms{L})$ is an L vector space.
 \begin{thm}
 $vec(\ms{L})$ is an L vector space.
 \end{thm}
 \begin{proof}
 We first show that $vec(\ms{L})$ is hereditarily Lindel\"of.  It suffices to prove that for any $\{(x_\alpha, O_\alpha): \alpha<\omega_1\}$ such that  each $O_\alpha$ is an open neighborhood of $x_\alpha$,   there are $\alpha<\beta$ such that $x_\beta\in O_\alpha$.
 
 Going to uncountable subset and shrinking neighborhoods, we may assume that there are $k, l<\omega$ and a sequence of rational intervals $\langle U_i: i<k\rangle$  such that for any $\alpha<\omega_1$, 
 $$x_\alpha=\sum_{j<l} r_\alpha(j) w_{b_\alpha(j)} \text{ and}$$
 $$O_\alpha=\{x: x(a_\alpha(i))\in U_i \text{ for any } i<k\}$$
 for some $a_\alpha\in [\omega_1]^k$, $b_\alpha\in [\omega_1]^l$ and $r_\alpha\in (\mathbb{R}\setminus\{0\})^l$.  Going to uncountable $\Gamma\subset \omega_1$, we may assume that 
 $$\ms{A}=\{a_\alpha: \alpha\in \Gamma\}, \ms{B}=\{b_\alpha: \alpha\in \Gamma\}$$ 
  are $\Delta$-systems with roots $a^r, b^r$ respectively and for any $\alpha\in \Gamma$, $a^r<a_\alpha\setminus a^r$, $b^r< b_\alpha\setminus b^r$.

  Applying Proposition \ref{prop1} to $\{a\setminus a^r: a\in \ms{A}\}, \{b\setminus b^r: b\in \ms{B}\}, \langle U_i: |a^r|\leq i<k\rangle $ and assignment $\{r'_{b_\alpha(j)}=r_\alpha(j): \alpha\in \Gamma, |b^r|\leq j <l\}$, we can find $\alpha<\beta$ in $\Gamma$ such that $b^r< a_\alpha\setminus a^r$ and for all $|a^r|\leq i<k$, 
$$\sum_{|b^r|\leq j<l} r_{\beta}(j)w_{b_\beta(j)}(a_\alpha(i)) \in U_i.$$
Now we check that $x_\beta\in O_\alpha$.

Fix $i<k$. If $i< |a^r|$, then $x_\beta(a_\alpha(i))=x_\beta(a_\beta(i))\in U_i$. If $i\geq |a^r|$, then $w_{b_\beta(j)}(a_\alpha(i))=0$ for $j< |b^r|$ by the fact that $b^r<a_\alpha\setminus a^r$. So 
$$x_\beta(a_\alpha(i))=\sum_{|b^r|\leq j<l} r_{\beta}(j)w_{b_\beta(j)}(a_\alpha(i)) \in U_i.$$

We then show that $vec(\ms{L})$ is not separable. Fix $X\in [vec(\ms{L})]^\omega$. Note for some $\alpha<\omega_1$, $x(\beta)=0$  whenever  $x\in X$ and $\beta>\alpha$. Then $w_\beta$ is not in the closure of $X$ whenever $\beta>\alpha$ since $w_\beta(\beta)=\theta_\beta\neq 0$.
 \end{proof}
 
 We now analyze the topological properties of $vec(\ms{L})^2$. We first get the following combinatorial property which follows from the proof of    \cite[Lemma 10]{pw}.
 \begin{prop}\label{prop2}
 For any $\Gamma\in [\omega_1]^{\omega_1}$, there are $A\in [\Gamma]^{\omega_1}$ and uncountable pairwise disjoint $\ms{B}\subset [\Gamma]^2$ such that for any $\alpha\in A$, for any $b\in \ms{B}\cap [\omega_1\setminus \alpha]^2$,
 $$w_{b(j)}(\alpha)>\sin 3  \text{ for some } j<2.$$
 \end{prop}
 \begin{proof}
Fix $\eta\in \Gamma$ so that $\theta_\eta$ is a complete accumulation point of $\{\theta_\xi:\xi\in \Gamma\}$. Fix $p<\omega$ such that
  $\frac{1}{3}<frac(p\theta_\eta )<\frac{1}{2}$. Choose an open neighbourhood $O$ of $\theta_\eta$ such that for any $x,y,z\in O$, $\frac{1}{3}<frac(px+y-z)<\frac{1}{2}$. Let $\Gamma'=\{\alpha\in \Gamma: \theta_\alpha\in O\}$.\medskip

By \cite[Theorem 9]{pw} (see also Theorem \ref{pw1} in Section 5), fix an uncountable set $A\subset \Gamma'$ and an uncountable pairwise disjoint family $\mathscr{B}\subset [\Gamma']^{2}$ such that for any $\alpha\in A$, for any $b\in \mathscr{B}\cap [\omega_1\setminus \alpha]^2$,
$$osc(\alpha,b(1))=osc(\alpha,b(0))+p.$$
 Then by our choice of $\Gamma'$, for any $\alpha\in A$, for any $b\in \mathscr{B}\cap [\omega_1 \setminus \alpha]^2$, there is some $j<2$ such that 
 $$frac(\theta_\alpha osc(\alpha,b(j))+\theta_{b(j)})\in (\frac{1}{3},1).$$
 To see this, assume that $c\stackrel{\triangle}{=}frac(\theta_\alpha osc(\alpha,b(0))+\theta_{b(0)})\leq \frac{1}{3}$. Then
 \begin{eqnarray*}
 && frac(\theta_\alpha osc(\alpha,b(1))+\theta_{b(1)})\\
  &=& frac(\theta_\alpha osc(\alpha,b(0))+\theta_{b(0)}+ p\theta_\alpha+\theta_{b(1)}-\theta_{b(0)})\\
 &= & frac(c+p\theta_\alpha+\theta_{b(1)}-\theta_{b(0)})\in (\frac{1}{3}, \frac{5}{6})
 \end{eqnarray*}
 
 Hence $w_{b(j)}(\alpha)> \sin 3$ for some $j<2$.
 \end{proof}

Recall  that the extent of a space $X$ is define as
$$e(X)=\sup\{|D|: D\subset X \text{ is closed and discrete } \}.$$
\begin{prop}\label{prop19}
  $e(vec(\mathscr{L})^2)=\omega_1$. Moreover, there is an uncountable subset $X\subset \mathscr{L}^2$ such that $X$ is closed discrete in $vec(\mathscr{L})^2$. In particular, $vec(\ms{L})^2$ is not Lindel\"of.
\end{prop}
\begin{proof}
  By Proposition \ref{prop2},  fix an uncountable set $A\subset \omega_1$ and an uncountable pairwise disjoint family $\mathscr{B}\subset [\omega_1]^{2}$ such that
  \begin{enumerate}
  \item for any $\alpha\in A$, for any $b\in \mathscr{B}\cap [\omega_1\setminus \alpha]^2$, there is some $j<2$, $w_{b(j)}(\alpha)> \sin 3$. 
  \end{enumerate}
  Going to uncountable subset,  assume moreover that  
    \begin{enumerate}\setcounter{enumi}{1}
  \item  for any $a\neq b$ in $\ms{B}$, either $a<b$ or $b<a$;
  \item for any $a<b$ in $\ms{B}$, there exists $\alpha \in A$ such that $a< \alpha <b$.
  \end{enumerate}
  
    Denote $X=\{w_b: b\in \mathscr{B}\}$ where $w_b=(w_{b(0)},w_{b(1)})$. In what follows, we show that $X$
  is closed discrete in $vec(\mathscr{L})^2$ and thus complete the proof. It suffices to prove that every point $(x,y)$ in $vec(\mathscr{L})^2$ has an open neighbourhood disjoint from
  $X\setminus\{(x,y)\}$.

  Now fix $(x,y)\in vec(\mathscr{L})^2$.  Let 
  $$\tau=\max\{\xi<\omega_1: x(\xi)\neq 0 \text{ or } y(\xi)\neq 0\}$$ 
  and $\tau=0$ if $x=y=\bf{0}$.
  Note by our definition of $w_\beta$'s, the maximal exists if one of $x$ and $y$ is not $\bf{0}$. 
  
  By (2), there is at most one $b\in \mathscr{B}$ such that  $\tau\in [b(0),b(1)]$ and $(x,y)\neq w_b$. Let
  $$O_1=vec(\mathscr{L})^2\setminus \{w_b: \tau\in [b(0), b(1)] \text{ and }(x,y)\neq w_b\}.$$
  
    Let  $$O_2=\{(s,t): s(\tau)\neq 0\text{ or } t(\tau)\neq 0\}\text{ if } \tau\neq 0\text{ and}$$
  $$O_2=vec(\mathscr{L})^2\text{ if }  \tau=0.$$
  
    Let $\gamma=\min (A\setminus (\tau+1))$ and $b_{min}$ be the element in $\mathscr{B}\cap [[\tau,\omega_1)]^2$ minimizing the first coordinate.
  $$O_3=\{(s,t): s(\gamma)<\sin 3\text{ and } t(\gamma)<\sin 3\}\setminus \{ w_{b_{min}}\}.$$
  
    Now we check that $O_1\cap O_2\cap O_3 $ is an open neighbourhood of $(x,y)$ disjoint from $X\setminus \{(x,y)\}$. Fix $w_b\in X\setminus \{(x,y)\}$.\medskip

  \textbf{Case 1}:  $\tau\in [b(0),b(1)]$.

  Then $w_b\notin O_1$.\medskip

  \textbf{Case 2}: $b(1)<\tau$.

  Then $w_b\notin O_2$.\medskip

  \textbf{Case 3}: $\tau<b(0)$.

  If $b=b_{min}$, then $w_b\notin O_3$. 
  
  Now suppose $b\neq b_{min}$. Then by (2), $b_{min}<b$. Together with  (3), we get that $b(0)>\gamma$. Hence by (1), $w_b\notin O_3$.  
\end{proof}

Recall that basic open sets of $(\mathbb{R}^{\omega_1})^2$ can be represented as finite intersections of $O_{\alpha,I,J}=\{(s,t): s(\alpha)\in I, t(\alpha)\in J\}$ where $\alpha<\omega_1$ and $I,J$ are open rational intervals. So there are only $\omega_1$ many basic open sets.
 
 The proof of Proposition 14 and Lemma 15 in \cite{pw} actually gives the following. 
 \begin{lem}[\cite{pw}]\label{lem20}
 Suppose that $X\subset \ms{L}^2$ is uncountable and closed discrete in $vec(\ms{L})^2$. Then the following two statements hold:
 \begin{enumerate}
 \item Suppose that $X=X_0\cup X_1$ is a partition such that 
 \begin{itemize}
 \item[$(\star)$] for any basic open set $U$, if $|U\cap X|=\omega_1$, then $  |U\cap X_0|=|U\cap X_1|=\omega_1$. 
 \end{itemize}
  If $U^0\supset X_0, U^1\supset X_1$ are open sets, then $U^0\cap U^1\cap \ms{L}^2\neq \emptyset$.
  
 \item If for each $x\in X$, $N_x$ is an open neighbourhood of $x$, then there is an infinite set $\{x_p: p<\omega\}\in [X]^\omega$ such that $\ms{L}^2\cap  \bigcap_{p<\omega} N_{x_p}\neq \emptyset$.
 \end{enumerate}
 \end{lem}
 For completeness, we sketch a proof here.
 \begin{proof}[Sketch of the proof.]
 (1) Let $M$ be a countable elementary submodel of $H(\aleph_2)$ containing everything relevant and  $\delta=M\cap \omega_1$. Fix $x\in X_0\setminus M$. Find  basic open set $U_0\in M$, $b\in [\omega_1\setminus M]^{k}$ and rational intervals $\langle I_{i,j}: i< k, j<2\rangle$ such that 
 $$x\in U_0\cap \bigcap_{i<k} O_{b(i), I_{i,0}, I_{i,1}} \subset U^0$$
 where $O_{\alpha,I,J}=\{(s,t): s(\alpha)\in I, t(\alpha)\in J\}$.
 By elementarity, we can find an uncountable pairwise disjoint $\ms{A}\subset [\omega_1]^k$ such that for any $a\in \ms{A}$,
 \begin{enumerate}[$(i)$]
 \item $U_0\cap \bigcap_{i<k} O_{a(i), I_{i,0}, I_{i,1}} \subset U^0.$
 \end{enumerate}
 
 By $(\star)$ and elementarity, $U_0\cap X_1$ is uncountable. Then for any $y\in U_0\cap X_1$, choose  basic open set $U_y=U_y'\times U_y''$ such that $y\in U_y\subset U_0\cap U^1$.
 
 Since $\ms{L}$ is hereditarily Lindel\"of, we can find $y\in U_0\cap X_1$ such that both $U_y'\cap \ms{L}$ and $U_y''\cap \ms{L}$ are uncountable. In other words,
 $$B_0=\{\alpha: w_\alpha\in U_y'\}\text{ and }B_1=\{\alpha: w_\alpha\in U_y''\}$$
 are uncountable.
 
 Applying Corollary \ref{cor13}, we get $a\in \ms{A}$ and $\beta_0\in B_0, \beta_1\in B_1$ such that for any $i<k$, 
 $w_{\beta_0}(a(i))\in I_{i,0} \text{ and } w_{\beta_1}(a(i))\in I_{i,1}.$
 Combining with $(i)$ and the fact that $(w_{\beta_0}, w_{\beta_1})\in U_y'\times U_y''\subset U_0\cap U^1$, we get $(w_{\beta_0}, w_{\beta_1})\in U^0\cap U^1$.\medskip
 
 (2)  Shrinking $N_x$ and going to uncountable subset of $X$, we may assume that for some basic open set $O$, $k<\omega$ and rational intervals $\langle I_{i,j}: i<k, j<2\rangle$,
 \begin{enumerate}[$(i)$]
 \item for any $x\in X$, there is $a_x\in [\omega_1]^k$ such that $N_x=O\cap \bigcap_{i<k} O_{a_x(i), I_{i,0}, I_{i,1}}$;
 \item $\ms{A}=\{a_x: x\in X\}$ is pairwise disjoint.
 \end{enumerate}
  Fix uncountable $B, B'\subset \omega_1$ such that $\{(w_\alpha, w_\beta): (\alpha,\beta)\in B\times B'\}\subset O$.
  
  Now find a sequence of rational intervals $\langle J_{i,j} \subset (0,1):i<k, j<2\rangle$ and $C\in [B]^{\omega_1}$, $C'\in [B']^{\omega_1}$ such that
for any $\beta\in C$, $\beta'\in C'$, $i<k$, $y\in J_{i,0}$, $z\in J_{i,1}$, $$f(frac(y+\theta_\beta))\in I_{i,0}\text{ and }f(frac(z+\theta_{\beta'}))\in I_{i,1}.$$

By Kronecker's Theorem, we can find a natural number $n$ and uncountable $\mathscr{A}'\subset \mathscr{A}$ such that for any $a\in \mathscr{A}'$, for any $c\in \mathbb{R}^k$, there are $m,m'<n$ such that $frac(m\theta_{a(i)}+c(i))\in J_{i,0}$ and $frac(m'\theta_{a(i)}+c(i))\in J_{i,1}$ for any $i<k$.\medskip

Passing to a pairwise disjoint subset of $C\times C'$, we can apply Lemma \ref{moore} to get $\{a_p: p<\omega\}\subset \mathscr{A}'$ and $\langle b_m=(\beta_m,\beta_m')\in C\times C':m<n\rangle$ such that $osc(a_p(i),b_m(j))=osc(a_p(i),b_0(j))+m$ for $p<\omega, i<k, j<2$ and $m<n$. Fix for each $p$, $x_p\in X$ such that $a_p=a_{x_p}$.

By the property of $\mathscr{A}'$ and $n$, for each $p<\omega$, there are $m,m'<n$ such that
$frac(\theta_{a_p(i)}osc(a_p(i),\beta_m))\in J_{i,0}$ and $frac(\theta_{a_p(i)}osc(a_p(i),\beta'_{m'}))\in J_{i,1}$ whenever $i<k$.  In other words, $(w_{\beta_m}, w_{\beta_{m'}'})\in N_{x_p}$.

So for some $m, m'<n$, $(w_{\beta_m}, w_{\beta_{m'}'})\in N_{x_p}$ for infinitely many $p$.
  \end{proof}
 
 Now the following is immediate.
 \begin{cor}\label{cor21}
 $vec(\ms{L})^2$ is neither normal nor weakly paracompact.
 \end{cor}
 \begin{proof}
 First by Proposition \ref{prop19}, fix an uncountable $X\subset \ms{L}^2$ that is closed discrete in $vec(\ms{L})^2$.
 
 Then we prove that $vec(\ms{L})^2$ is not normal.    Let $X=X_0\cup X_1$ be a partition satisfying property $(\star)$ in Lemma \ref{lem20}. Then by Lemma \ref{lem20}, $X_0$ and $X_1$ are two disjoint closed sets that can not be separated by disjoint  open sets.
 
 Finally we prove that $vec(\ms{L})^2$ is not weakly paracompact.  For each $x\in X$, choose   an open neighbourhood $O_x$ of $x$ such that
$O_x\cap X=\{x\}$. By Lemma \ref{lem20}, $\{vec(\mathscr{L})^2\setminus X\}\cup \{O_x: x\in X\}$ is an open cover without  a point-finite open refinement. To see this, just note that every open refinement has a subset of form
$\{O_x'\subset O_x:$ $x\in X\}$ where $O_x'$ is an open neighbourhood of $x$.
 \end{proof}

 \section{Topological field}

A \emph{topological field} is a field endowed with a topology with respect to which all algebraic  operations are continuous.  

 In this section we will   construct an \emph{L field}---a topological field that is an L space---whose square is neither normal nor weakly paracompact.\medskip

 \begin{defn}
  Throughout this section, we fix a function $g: (0,1] \ra \mathbb{C}$ such that
  \begin{enumerate}
  \item $g(x)=x$ for $x\in [\frac{1}{3}, 1]$;
  \item for any integer $n\geq 3$, $g$ is a linear function on $[\frac{1}{n+1}, \frac{1}{n}]$;
  \item for any positive integer $n$, $g(\frac{1}{n})=p+qi$ for some rational number $p, q$;
  \item for any positive integer $n$, $g[(0, \frac{1}{n}]]$ is dense in $\mathbb{C}$. 
  \end{enumerate}
 
   \end{defn}
   
   We now fix an algebraically  independent set of reals $\{\theta_\alpha, \theta_\alpha':\alpha<\omega_1\}$ with the property ganranteed by the following lemma.

   \begin{lem}\label{nr}
   There is an algebraically  independent set of reals $\{\theta_\alpha, \theta_\alpha':\alpha<\omega_1\} $ such that
   for any $l<\omega$, for any $b\in [\omega_1]^l$, for any polynomials $P(z_0,...z_{l-1})$, $Q(z_0,...z_{l-1})$ with rational coefficients,
   $$\frac{P(\theta_{b(0)}+\theta_{b(0)}'i,...,\theta_{b(l-1)}+\theta_{b(l-1)}'i)}{Q(\theta_{b(0)}+\theta_{b(0)}'i,...,\theta_{b(l-1)}+\theta_{b(l-1)}'i)}\notin \mathbb{R} \eqno(NR)$$
   provided that $\frac{P(z_0,...z_{l-1})}{Q(z_0,...z_{l-1})}$ is not a constant.
   \end{lem}
   In other words, $Im(\frac{P(\theta_{b(0)}+\theta_{b(0)}'i,...,\theta_{b(l-1)}+\theta_{b(l-1)}'i)}{Q(\theta_{b(0)}+\theta_{b(0)}'i,...,\theta_{b(l-1)}+\theta_{b(l-1)}'i)})\neq 0$ where $Im(z)$ is the imaginary part of $z$. To prove the lemma, we will need the following  special case of Kuratowski-Ulam Theorem.

\begin{thm}[\cite{ku}]\label{ku}
  For any comeager $X\subset \mathbb{R}^{2}$, $\{x\in \mathbb{R}: X_x$ is comeager$\}$ is comeager where $X_x=\{y: (x,y)\in X\}$.
\end{thm}
   
\begin{proof}[Proof of Lemma \ref{nr}.]
 We choose $\theta_\alpha, \theta_\alpha'$ inductively. At step $\alpha$, assume $\{\theta_\xi, \theta_\xi': \xi<\alpha\}$ has been constructed. 
 
 Let $\{\frac{P_n(z)}{Q_n(z)}: n<\omega\}$ enumerate all non-constant  rational functions (with single variable) with coefficients in the algebraic closure of $\{\theta_\xi, \theta'_\xi: \xi<\alpha\}$.
 
 First fix $n<\omega$ and denote $f_n(z)=\frac{P_n(z)}{Q_n(z)}$. Note that $f_n$ is analytic on all but finitely many points in $\mathbb{C}$. Then $f_n^{-1}[\mathbb{R}]$ is closed (in $dom(f_n)$) and by the Open Mapping Theorem,  has empty interior.  This shows that $f_n^{-1}[\mathbb{R}]$ is nowhere dense.
 
 Now $A_\alpha=\bigcup_{n<\omega} f_n^{-1}[\mathbb{R}]$ is meager. Then $X=\mathbb{C}\setminus A_\alpha$ is comeager. By Theorem \ref{ku}, we can  find $\theta_\alpha$ such that $\{\theta_\xi, \theta_\xi': \xi<\alpha\}\cup \{\theta_\alpha\}$  is algebraically independent and $X_{\theta_\alpha}$ is comeager. Then find $\theta_\alpha'\in X_{\theta_\alpha}$ such that $\{\theta_\xi, \theta_\xi': \xi\leq \alpha\} $  is algebraically independent. This finishes the induction at  $\alpha$ and hence the induction.\medskip
 
 We now check that $(NR)$ is satisfied.  Arbitrarily choose $l<\omega$,   $b\in [\omega_1]^l$,   polynomials $P(x_0,...x_{l-1})$, $Q(x_0,...x_{l-1})$ with rational coefficients such that $\frac{P(z_0,...z_{l-1})}{Q(z_0,...z_{l-1})}$ is not a constant. Let $f(z)=\frac{P(\theta_{b(0)}+\theta_{b(0)}'i,...,\theta_{b(l-2)}+\theta_{b(l-2)}'i, z)}{Q(\theta_{b(0)}+\theta_{b(0)}'i,...,\theta_{b(l-2)}+\theta_{b(l-2)}'i, z)}$. First note that $\theta_{b(l-1)}+\theta'_{b(l-1)}i\in dom (f)$ by algebraic independence. Then note by our construction, $\theta_{b(l-1)}+\theta'_{b(l-1)}i\notin A_{b(l-1)}$. So $f(\theta_{b(l-1)}+\theta'_{b(l-1)}i)\notin \mathbb{R}$. This shows that $(NR)$ is satisfied for $l, b , P, Q$. Since they are arbitrarily chosen, $(NR)$ is satisfied.
   \end{proof}
   
  \begin{defn}
   \begin{enumerate}
\item $\mathcal{L}=\{v_\beta: \beta<\omega_1\}$ where
\begin{displaymath}
   v_\beta(\alpha) = \left\{
     \begin{array}{lr}
       g(frac(\theta_\alpha osc(\alpha,\beta)+\theta_\beta))  & : \alpha<\beta\\
       \theta_\beta & : \alpha=\beta\\
       \theta_\beta+\theta_\beta'i & : \alpha> \beta.
     \end{array}
   \right.
\end{displaymath}
We view $\mathcal{L}$ as a subspace of $\mathbb{C}^{\omega_1}$, equipped with the product topology.
\item $field(\mc{L})$ is the topological field   generated by $\mc{L}$, where the addition $+$ and multiplication $\cdot$ are coordinatewise addition and coordinatewise multiplication.
\end{enumerate}
  \end{defn}
  
  We will show that $field(\mc{L})$ is an  L field. As in the topological vector space case, we will prove a property which is stronger than   hereditary Lindel\"ofness. We first fix a notation. For any $\alpha\leq \omega_1$, 
  $$K_\alpha\subset \mathbb{C}\text{ is the field generated by }\mathbb{Q}+\mathbb{Q}i\cup \{\theta_\xi, \theta_\xi': \xi<\alpha\}.$$
  
  \begin{prop}\label{prop23}
Let  $\mathscr{A}\subset [\omega_1]^k$ and
    $\mathscr{B}\subset [\omega_1]^l$ be uncountable families of pairwise disjoint sets. Then for every sequence of non-empty open sets $\langle U_s\subset\mathbb{C}: s<k \rangle$ and every non-constant rational function $\frac{P(z_0,...z_{l-1})}{ Q(z_0,...z_{l-1})}$ with   coefficients in $K_{\omega_1}$, there are $a\in \mathscr{A}, b\in \mathscr{B}$ such that $a<b$ and for all $s<k$, 
    $$\frac{P(v_{b(0)}(a(s)),..., v_{b(l-1)}(a(s)))}{Q(v_{b(0)}(a(s)),..., v_{b(l-1)}(a(s)))}\in U_s.$$ 
\end{prop}
\begin{proof}
  Going to uncountable subfamilies we may assume, by \cite[Theorem 7]{pw} (or Theorem \ref{pw} in Section 5), that there is a sequence $\langle c_{sj}: s<k,j<l\rangle \in \mathbb{Z}^{k\times l}$ such that for any $a\in \mathscr{A}$, for any
    $b\in \mathscr{B}$, if $a<b$, then $osc(a(s),b(j))=osc(a(s),b(0))+c_{sj}$ whenever $s<k,j<l$. Note that for any $s<k$, $$c_{s0}=0.$$
    
   Fix $\eta<\omega_1$ such that the coefficients of $\frac{P(z_0,...z_{l-1})}{ Q(z_0,...z_{l-1})}$ are in $K_\eta$. Assume that  $\eta<a$ whenever $a\in \ms{A}\cup \ms{B}$.

Without loss of generality, assume $z_0$ is free in $\frac{P(z_0,...z_{l-1})}{ Q(z_0,...z_{l-1})}$, i.e., $\frac{P(z_0,...z_{l-1})}{ Q(z_0,...z_{l-1})}\neq f(z_1,..., z_{l-1})$ for any rational function $f$.

Fix  a countable elementary submodel $M\prec H(\aleph_2)$ containing everything relevant. Fix for now $a\in \ms{A}\setminus M$, $b\in \ms{B}\setminus M$ and $s<k$.
Denote for $0<j<l$,
$$r_j=g(frac(-\theta_{b(0)}+\theta_{a(s)}c_{sj}+\theta_{b(j)})).$$
Note by our choice of $g$, $\ms{A}, \ms{B}$ and $\theta_\alpha$'s, $r_1,...r_{l-1}$ are algebraically independent over $K_\eta$. 

Consider 
    $$F_s(z)=\frac{P(z, r_1,..., r_{l-1})}{Q(z, r_1,..., r_{l-1})}.$$

    \textbf{Claim.} $F_s(z)$ is not a constant. 
    \begin{proof}[Proof of Claim.]
    Assume 
    $$P(z_0,...,z_{l-1})=\sum_{t\leq n} P_t(z_1,...,z_{l-1})z_0^t\text{ and }Q(z_0,...,z_{l-1})=\sum_{t\leq n'} Q_t(z_1,...,z_{l-1})z_0^t.$$
    Suppose towards a contradiction that $F_s(z)=c$ for some constant $c$. Then $n=n'$ and for any $t\leq n$, $P_t(r_1,...,r_{l-1})=cQ_t(r_1,...,r_{l-1})$. 
    
    By the fact that $r_1,...r_{l-1}$ are algebraically independent over $K_\eta$ and $P_t, Q_t$ are polynomials with   coefficients in $K_\eta$, 
    $P_n(r_1,..., r_{l-1})\neq 0$, $Q_n(r_1,..., r_{l-1})\neq 0$ and $c=\frac{P_n(r_1,..., r_{l-1})}{Q_n(r_1,..., r_{l-1})}\neq 0$.

    For any $t\leq n$,   $P_t(r_1,..., r_{l-1})=cQ_t(r_1,..., r_{l-1})=\frac{P_n(r_1,..., r_{l-1})}{Q_n(r_1,..., r_{l-1})}Q_t(r_1,..., r_{l-1})$. 
    By algebraic independence again, $P_t(z_1,..., z_{l-1})=\frac{P_n(z_1,..., z_{l-1})}{Q_n(z_1,..., z_{l-1})}Q_t(z_1,..., z_{l-1})$.
    
    But this shows that $\frac{P(z_0,...,z_{l-1})}{Q(z_0,...,z_{l-1})}=\frac{P_n(z_1,..., z_{l-1})}{Q_n(z_1,..., z_{l-1})}$. This is a contradiction to our assumption that  $z_0$ is free in $\frac{P(z_0,...z_{l-1})}{ Q(z_0,...z_{l-1})}$.    
    \end{proof}
    
    Now by above claim and the Fundamental Theorem of Algebra, $F_s(z)\in U_s$ has a solution. By continuity, we can find open sets $\langle O'_{sj}\subset \mathbb{C}: j<l\rangle $ such that 
    \begin{enumerate}[$(i)$]
    \item for any $0<j<l$, $r_j\in O'_{sj}$;
    \item for any $\langle z_j\in O'_{sj}: j<l\rangle$, $\frac{P(z_0,...,z_{j-1})}{Q(z_0,...,z_{j-1})}\in U_s$.
    \end{enumerate}
    Now for any $0<j<l$, define $G_{sj}$ to be the function with domain $dom(G_{sj})=\mathbb{R}\setminus (-\theta_{a(s)}c_{sj}-\theta_{b(j)}+\mathbb{Z})$ by
    $$G_{sj}(z)=g(frac(z+\theta_{a(s)}c_{sj}+\theta_{b(j)})).$$
    Note as $z\ra -\theta_{b(0)}$, $G_{sj}(z)\ra r_j$ for $0<j<l$.
    
    We first find by continuity,   $O_s $, $V_s'$ and $\langle V_{sj}'': 0<j<l\rangle$ that are open subsets of $\mathbb{R}$ such that 
    \begin{enumerate}[$(i)$]\setcounter{enumi}{2}
    \item $-\theta_{b(0)}\in O_s$, $\theta_{a(s)}\in V_s'$ and for any $0<j<l$, $\theta_{b(j)}\in V_{sj}''$;
    \item for any $z\in O_s$,   $x\in V_s'$, $0<j<l$ and $ y_j\in V_{sj}''$, $g(frac(z+xc_{sj}+y_j))\in O_{sj}'$. 
    \end{enumerate}
    
    Note as $z\ra -\theta_{b(0)}^+$,  $frac(z+\theta_{a(s)}c_{s0}+\theta_{b(0)})=frac(z+ \theta_{b(0)})\ra 0^+$.
    
    We then find by     property (4) in the definition of $g$,    open sets $V^*_s\subset O_s$ and $V_{s0}''$  such that 
    \begin{enumerate}[$(i)$]\setcounter{enumi}{4}
    \item $\theta_{b(0)}\in V_{s0}''$;
    \item for any $z\in V^*_s$,  for any $y_0\in V_{s0}''$,  $g(frac(z+y_0))\in O'_{s0}$. 
     \end{enumerate}
       Let $V_s=frac(V_s^*)=\{frac(z): z\in V^*_s\}$. Note 
       $$frac(z+xc_{sj}+y_j)=frac(frac(z)+xc_{sj}+y_j).$$
   Together with $(iii)-(vi)$, we get the following.
    \begin{enumerate}[$(i)$]\setcounter{enumi}{6}
    \item $\theta_{a(s)}\in V_s'$ and for any $j<l$, $\theta_{b(j)}\in V_{sj}''$;
    \item for any $z\in V_s$,   $x\in V_s'$, $j<l$ and $ y_j\in V_{sj}'' $, $g(frac(z+xc_{sj}+y_j))\in O_{sj}'$.
     \end{enumerate}
     Together with $(ii)$, we get the following.
     \begin{enumerate}[$(i)$]\setcounter{enumi}{8}
     \item for any $z\in V_s$,   $x\in V_s'$  and $\langle y_j\in V_{sj}'': j<l\rangle $,
     $$\frac{P(g(frac(z+xc_{s0}+y_0)),...,g(frac(z+xc_{s,l-1}+y_{l-1})))}{Q(g(frac(z+xc_{s0}+y_0)),...,g(frac(z+xc_{s,l-1}+y_{l-1})))}\in U_s.$$
     \end{enumerate}
     
     Repeat above argument for each $s<k$. Let $V_{j}''=\bigcap_{s<k} V_{sj}''$ for each $j<l$. Then we get the following.
     \begin{enumerate}[$(i)$]\setcounter{enumi}{9}
     \item for any $s<k$ and $j<l$, $\theta_{a(s)}\in V_s'$ and   $\theta_{b(j)}\in V_{j}''$;
          \item for any $s<k$, $z\in V_s$,   $x\in V_s'$  and $\langle y_j\in V_{j}'': j<l\rangle $,
     $$\frac{P(g(frac(z+xc_{s0}+y_0)),...,g(frac(z+xc_{s,l-1}+y_{l-1})))}{Q(g(frac(z+xc_{s0}+y_0)),...,g(frac(z+xc_{s,l-1}+y_{l-1})))}\in U_s.$$
     \end{enumerate}


So $\underset{s<k}{\prod} V_s'$ and $\underset{j<l}{\prod} V_j''$  are open neighbourhoods of $(\theta_{a(0)},...,\theta_{a(k-1)})$ and $(\theta_{b(0)},...,\theta_{b(l-1)})$   respectively.
    By elementarity, 
    $$\ms{A}'=\{a'\in \ms{A}: (\theta_{a'(0)},...,\theta_{a'(k-1)})\in  \underset{s<k}{\prod} V_s'\},$$
    $$\ms{B}'=\{b'\in \ms{B}: (\theta_{b'(0)},...,\theta_{b'(k-1)})\in  \underset{j<l}{\prod} V_j'' \}$$
    are uncountable.\smallskip

    Finally, by Kronecker's Theorem, we can find an uncountable subfamily $\mathscr{A}''\subset \mathscr{A}'$ and a natural number $m$ such that for any $a\in \mathscr{A}''$ and
    $\sigma\in \mathbb{R}^k$, there is a $m'<m$ such that for any $s<k$, $frac(\theta_{a(s)}m'+\sigma(s))\in V_s$. Then by Lemma \ref{moore}, we can find $a\in \mathscr{A}''$ and $\{b_t\in \mathscr{B}': t<m\}$ such
    that $osc(a(s),b_t(0))=osc(a(s),b_0(0))+t$ for any $s<k,t<m$. 
    
    Now we can find some $t<m$ such that for any $s<k$,
    $$frac(\theta_{a(s)}osc(a(s),b_t(0)))\in V_s.$$
    We now check that $a, b_t$ are as desired. Fix $s<k$. Let $z=frac(\theta_{a(s)}osc(a(s),b_t(0)))$. 
    Recall 
    \begin{eqnarray*}
  & &  frac(\theta_{a(s)}osc(a(s),b_t(j))+\theta_{b_t(j)})\\
   &= &frac(\theta_{a(s)}osc(a(s),b_t(0))+\theta_{a(s)}c_{sj}+\theta_{b_t(j)})\\
    & = & frac(frac(\theta_{a(s)}osc(a(s),b_t(0)))+\theta_{a(s)}c_{sj}+\theta_{b_t(j)})\\
    & =& frac(z+\theta_{a(s)}c_{sj}+\theta_{b_t(j)})
     \end{eqnarray*}
    By $(xi)$ and our choice of $\ms{A}', \ms{B}'$,  
    \begin{eqnarray*}
 && \frac{P(v_{b_t(0)}(a(s)),..., v_{b_t(l-1)}(a(s)))}{Q(v_{b_t(0)}(a(s)),..., v_{b_t(l-1)}(a(s)))}\\
  &=&\frac{P(g(frac(z+\theta_{a(s)}c_{s0}+\theta_{b_t(0)})),...,g(frac(z+\theta_{a(s)}c_{s,l-1}+\theta_{b_t(l-1)})))}{Q(g(frac(z+\theta_{a(s)}c_{s0}+\theta_{b_t(0)})),...,g(frac(z+\theta_{a(s)}c_{s,l-1}+\theta_{b_t(l-1)})))}\in U_s.
 \end{eqnarray*}
 This finishes the proof of the proposition.
\end{proof}

 As before, the topological field $field(\mc{L})$ is an L   field.
 \begin{thm}
 $field(\mc{L})$ is an L field.
 \end{thm}
 \begin{proof}
 We first show that $field(\mc{L})$ is hereditarily Lindel\"of.  It suffices to prove that for any $\{(x_\alpha, O_\alpha): \alpha<\omega_1\}$ such that each $O_\alpha$ is an open neighborhood of $x_\alpha$, there are $\alpha<\beta$ such that $x_\beta\in O_\alpha$.
 
 Going to uncountable subset and shrinking neighborhoods, we may assume that there are $k, l<\omega$, rational function $\frac{P(z_0,...,z_{l-1})}{Q(z_0,...,z_{l-1})}$ with rational coefficients and a sequence of open sets $\langle U_s: s<k\rangle$  such that for any $\alpha<\omega_1$, 
 $$x_\alpha=\frac{P(v_{b_\alpha(0)},...,v_{b_\alpha(l-1)})}{Q(v_{b_\alpha(0)},...,v_{b_\alpha(l-1)})} \text{ and}$$
 $$O_\alpha=\{x: x(a_\alpha(s))\in U_s \text{ for any } s<k\}$$
 for some $a_\alpha\in [\omega_1]^k$ and $b_\alpha\in [\omega_1]^l$.  Going to uncountable $\Gamma\subset \omega_1$, we assume that 
 $$\ms{A}=\{a_\alpha: \alpha\in \Gamma\}, \ms{B}=\{b_\alpha: \alpha\in \Gamma\}$$ 
  are $\Delta$-systems with roots $a^r, b^r$. 
  
  Define
  $$\frac{P^*(z_{|b^r|},...,z_{l-1})}{Q^*(z_{|b^r|},...,z_{l-1})}=\frac{P(\theta_{b^r(0)}+\theta'_{b^r(0)}i, ..., \theta_{b^r(|b^r|-1)}+\theta'_{b^r(|b^r|-1)}i, z_{|b^r|},...,z_{l-1})}{Q(\theta_{b^r(0)}+\theta'_{b^r(0)}i, ..., \theta_{b^r(|b^r|-1)}+\theta'_{b^r(|b^r|-1)}i, z_{|b^r|},...,z_{l-1})}.$$
  
  Applying Proposition \ref{prop23} to $\{a\setminus a^r: a\in \ms{A}\}, \{b\setminus b^r: b\in \ms{B}\}, \langle U_s: |a^r|\leq s<k\rangle $ and rational function $\frac{P^*(z_{|b^r|},...,z_{l-1})}{Q^*(z_{|b^r|},...,z_{l-1})}$ with coffecients in $K_{\omega_1}$, we get $\alpha<\beta$ in $\Gamma$ such that $ b^r< a_\alpha\setminus a^r$ and for any $|a^r|\leq s<k$, 
$$\frac{P^*(v_{b_\beta(|b^r|)}(a_\alpha(s)),...,v_{b_\beta(l-1)}(a_\alpha(s)))}{Q^*(v_{b_\beta(|b^r|)}(a_\alpha(s)),...,v_{b_\beta(l-1)}(a_\alpha(s)))} \in U_s.$$
Now it is straightforward to check that $x_\beta\in O_\alpha$.\medskip

We then show that $field(\mc{L})$ is not separable. Fix $X\in [field(\mc{L})]^\omega$. Note for some $\alpha<\omega_1$, $x(\beta)=x(\beta+1)$  whenever  $x\in X$ and $\beta>\alpha$. By Proposition \ref{prop23}, we can find $y\in \mc{L}$ and $\beta>\alpha$ such that $y(\beta)\neq y(\beta+1)$. Then this $y$ is not in the closure of $X$.
 \end{proof}

 We now analyze the topological properties of $field(\mc{L})^2$. As in Proposition \ref{prop2}, we   get the following combinatorial property.
 \begin{prop}\label{prop25}
 For any $\Gamma\in [\omega_1]^{\omega_1}$, there are $A\in [\Gamma]^{\omega_1}$ and uncountable pairwise disjoint $\ms{B}\subset [\Gamma]^2$ such that for any $\alpha\in A$, for any $b\in \ms{B}\cap [\omega_1\setminus \alpha]^2$,
 $$v_{b(j)}(\alpha)\in (\frac{1}{3}, 1) \text{ for some } j<2$$
 where $(\frac{1}{3}, 1)=\{x\in \mathbb{R}: \frac{1}{3}<x<1\}$.
 \end{prop}
 \begin{proof}
Fix $\eta\in \Gamma$ so that $\theta_\eta$ is a complete accumulation point of $\{\theta_\xi:\xi\in \Gamma\}$. Fix $p<\omega$ such that
  $\frac{1}{3}<frac(p\theta_\eta )<\frac{1}{2}$. Choose an open neighbourhood $O$ of $\theta_\eta$ such that for any $x,y,z\in O$, $\frac{1}{3}<frac(px+y-z)<\frac{1}{2}$. Let $\Gamma'=\{\alpha\in \Gamma: \theta_\alpha\in O\}$.\medskip

By Theorem \ref{pw1}, fix an uncountable set $A\subset \Gamma'$ and an uncountable pairwise disjoint family $\mathscr{B}\subset [\Gamma']^{2}$ such that for any $\alpha\in A$, for any $b\in \mathscr{B}\cap [\omega_1\setminus \alpha]^2$,
$$osc(\alpha,b(1))=osc(\alpha,b(0))+p.$$
 Then by our choice of $\Gamma'$, for any $\alpha\in A$, for any $b\in \mathscr{B}\cap [\omega_1\setminus \alpha]^2$, there is some $j<2$ such that 
 $$frac(\theta_\alpha osc(\alpha,b(j))+\theta_{b(j)})\in (\frac{1}{3},1)$$
 and hence $v_{b(j)}(\alpha)\in (\frac{1}{3}, 1)$.
 \end{proof}

\begin{prop}\label{prop26}
  $e(field(\mc{L})^2)=\omega_1$. Moreover, there is an uncountable subset $X\subset \mc{L}^2$ such that $X$ is closed discrete in $field(\mc{L})^2$.
\end{prop}
\begin{proof}
  By Proposition \ref{prop25}, we can find an uncountable set $A\subset \omega_1$ and an uncountable pairwise disjoint family $\mathscr{B}\subset [\omega_1]^{2}$ such that
  \begin{enumerate}
  \item for any $\alpha\in A$, for any $b\in \mathscr{B}\cap [\omega_1\setminus \alpha]^2$, there is some $j<2$, $v_{b(j)}(\alpha) \in (\frac{1}{3},1)$. 
  \end{enumerate}
  Going to uncountable subsets, we assume moreover that  
    \begin{enumerate}\setcounter{enumi}{1}
  \item  for any $a\neq b$ in $\ms{B}$, either $a<b$ or $b<a$;
  \item for any $a<b$ in $\ms{B}$, there exists $\alpha \in A$ such that $a< \alpha <b$.
  \end{enumerate}
  
  Let $X=\{v_b: b\in\ms{B}\}$ where $v_b=(v_{b(0)}, v_{b(1)})$.   In what follows, we show that $X$
  is closed discrete in $field(\mc{L})^2$ and thus complete the proof. It suffices to prove that every point $(x,y)$ in $field(\mc{L})^2$ has an open neighbourhood disjoint from
  $X\setminus\{(x,y)\}$. 
  
  First for each $x\in field(\mc{L})^2\setminus \{\bf{0}\}$, fix $a_x\in [\omega_1]^{<\omega}$ and rational functions $F_x(z_0,...,z_{|a_x|-1})$ such that
   \begin{enumerate}\setcounter{enumi}{3}
  \item for any $s< |a_x|$, $z_s$ is free in $F_x(z_0,...,z_{|a_x|-1})$;
  \item $x=F_x(v_{a_x(0)},..., v_{a_x(|a_x|-1)})$.
  \end{enumerate}
  
  Now fix $(x,y)\in field(\mc{L})^2$.  Let 
  $$\tau=\max(a_x\cup a_y)$$ 
  and denote $a_{\bf{0}}=\emptyset$ and $\max \emptyset=0$.
  
  By (2), there is at most one $b\in \mathscr{B}$ such that  $\tau\in [b(0),b(1)]$ and $(x,y)\neq v_b$. Let
  $$O_1=field(\mc{L})^2\setminus \{v_b: b\in \ms{B}, \tau\in [b(0), b(1)] \text{ and } (x,y)\neq v_b\}.$$
  
    Let  $$O_2=\{(s,t): s(\tau)\neq s(\tau+1)\text{ or } t(\tau)\neq t(\tau+1)\}\text{ if } \tau\neq 0\text{ and}$$
  $$O_2=field(\mc{L})^2\text{ if }  \tau=0.$$
  
    Let $\gamma=\min (A\setminus (\tau+1))$ and $b_{min}$ be the element in $\mathscr{B}\cap [[\tau,\omega_1)]^2$ minimizing the first coordinate.
  $$O_3=\{(s,t): s(\gamma)\notin [\frac{1}{3},1]\text{ and } t(\gamma)\notin [\frac{1}{3},1]\}\setminus \{ v_{b_{min}}\}.$$
  
    Now we check that $O_1\cap O_2\cap O_3 $ is an open neighborhood of $(x,y)$ disjoint from $X\setminus \{(x,y)\}$. 
    
    We first check that $O_2, O_3$ are neighborhoods of $(x,y)$. For $O_2$, assume without loss of generality that $a_x\neq \emptyset$ and $\tau=\max a_x$. Then 
    $$x(\tau)=F_x(\theta_{a_x(0)}+\theta_{a_x(0)}'i,...,\theta_{a_x(|a_x|-2)}+\theta_{a_x(|a_x|-2)}'i, \theta_{a_x(|a_x|-1)}) \text{ and}$$
    $$x(\tau+1)=F_x(\theta_{a_x(0)}+\theta_{a_x(0)}'i,...,\theta_{a_x(|a_x|-2)}+\theta_{a_x(|a_x|-2)}'i, \theta_{a_x(|a_x|-1)}+\theta_{a_x(|a_x|-1)}'i) .$$
    So $x(\tau)\neq x(\tau+1)$ follows from algebraic independence and (4). For $O_3$, by symmetry, we only check $x$.  $x=\bf{0}$ is trivial and $x\neq \bf{0}$ follows from Lemma \ref{nr}.
    
    We then fix $v_b\in X\setminus \{(x,y)\}$.\medskip

  \textbf{Case 1}:  $\tau\in [b(0),b(1)]$.

  Then $v_b\notin O_1$.\medskip

  \textbf{Case 2}: $b(1)<\tau$.

  Then $v_b\notin O_2$.\medskip

  \textbf{Case 3}: $\tau<b(0)$.

  If $b=b_{min}$, then $v_b\notin O_3$. 
  
  Now suppose $b\neq b_{min}$. Then by (2), $b_{min}<b$. Together with  (3), we get that $b(0)>\gamma$. Hence by (1), $v_b\notin O_3$.  
\end{proof}

The proof of Lemma \ref{lem20} actually  gives the following.
 \begin{lem}\label{lem27}
 Suppose $X\subset \mc{L}^2$ is uncountable and closed discrete in $filed(\mc{L})^2$. Then the following two statements hold:
 \begin{enumerate}
 \item Suppose $X=X_0\cup X_1$ is a partition such that 
\begin{itemize}
\item[$(\star)$] for any basic open set $U$, if $|U\cap X|=\omega_1$, then $  |U\cap X_0|=|U\cap X_1|=\omega_1$.
\end{itemize}
 If $U^0\supset X_0, U^1\supset X_1$ are open sets, then $U^0\cap U^1\cap \mc{L}^2\neq \emptyset$.
  
 \item If for each $x\in X$, $N_x$ is an open neighbourhood of $x$, then there is an infinite $\{x_p: p<\omega\}\in [X]^\omega$ such that $\mc{L}^2\cap  \bigcap_{p<\omega} N_{x_p}\neq \emptyset$.
 \end{enumerate}
 \end{lem}

Now the following is a consequence of Proposition \ref{prop26} and Lemma \ref{lem27}.
 \begin{cor}\label{cor28}
 $field(\mc{L})^2$ is neither normal nor weakly paracompact.
 \end{cor}

  \section{More on combinatorial properties of $osc$}
  The results in \cite{pw} show that the $osc$ mapping is highly inhomogeneous. In this section, we show that this sort of inhomogeneous phenomenon could even be recognized homogeneously.

  First recall     \cite[Theorem 7]{pw}.
    \begin{thm}[\cite{pw}]\label{pw}
    For any uncountable families of pairwise disjoint sets $\mathscr{A}\subset [\omega_1]^k$ and
    $\mathscr{B}\subset [\omega_1]^l$, there are $\mathscr{A}'\in [\mathscr{A}]^{\omega_1}$, $\mathscr{B}'\in [\mathscr{B}]^{\omega_1}$ and
    $\langle c_{ij}: i<k, j<l\rangle \in \mathbb{Z}^{k\times l}$ such that for any $a\in \mathscr{A}'$, for any
    $b\in \mathscr{B}'$, if $a<b$, then $osc(a(i),b(j))=osc(a(i),b(0))+c_{ij}$ for any $i<k,j<l$.
  \end{thm}
  On the other hand, Theorem 9 in \cite{pw} indicates that   every  appropriate matrix $\langle c_{ij}: i<k,j<l\rangle$ can be realized.
\begin{thm}[\cite{pw}]\label{pw1}
  For any $X\in [\omega_1]^{\omega_1}$, for any $\langle c_{ij}: i<k,j<l\rangle\in \mathbb{Z}^{k\times l}$ such that $c_{i0}=0$ for any $i<k$, there are uncountable
  families $\mathscr{A}\subset [X]^k$, $\mathscr{B}\subset [X]^l$ that are pairwise disjoint and such that for any $a\in \mathscr{A},b\in \mathscr{B}$, if $a<b$, then $osc(a(i),b(j))=osc(a(i),b(0))+c_{ij}$ for any $i<k,j<l$.
\end{thm}
  
  The proof of   \cite[Theorem 7]{pw} actually shows that we can take $\mathscr{A}'=\mathscr{B}'$ if $\mathscr{A}=\mathscr{B}$.
 \begin{cor}\label{cor1}
    For any uncountable family of pairwise disjoint sets $\mathscr{A}\subset [\omega_1]^k$, there are $\mathscr{A}'\in [\mathscr{A}]^{\omega_1}$ and
    $\langle c_{ij}: i, j<k\rangle \in \mathbb{Z}^{k\times k}$ such that for any $a,b$ in $\mathscr{A}'$, if $a<b$, then $osc(a(i),b(j))=osc(a(i),b(0))+c_{ij}$ for any $i,j<k$.
  \end{cor}
  \textbf{Observation.} If we look into the proof of \cite[Theorem 7]{pw} closely, then we can observe that in Corollary \ref{cor1}, for each $i<k$,  $c_{ii}=\min\{c_{ij}: j<k\}$. The reason is that an initial part of $\varrho_{1a(i)}$ agrees with the initial part of $\varrho_{1b(i)}$. And consequently,  $\varrho_{1a(i)}$ and $\varrho_{1b(i)}$ admit no oscillation on an initial part of $L(a(i), b(i))$ while the oscillation between $\varrho_{1a(i)}$ and $\varrho_{1b(i)}$ on the rest tail part of $L(a(i), b(i))$ agrees with the oscillation between $\varrho_{1a(i)}$ and $\varrho_{1b(j)}$ on a tail part of $L(a(i), b(j))$. So if we define $c_{ij}'\stackrel{\triangle}{=}c_{ij}-c_{ii}$,  then each $c_{ij}'\geq 0$ and in Corollary \ref{cor1}, we obtain $osc(a(i),b(j))=osc(a(i),b(i))+c_{ij}'$ for any $i,j<k$.

  Thus Corollary \ref{cor1} provides a homogeneous set with the same pattern of inhomogeneity. Like the relationship between Theorem \ref{pw} and Theorem \ref{pw1}, it is also interesting to get a homogeneous set with the same prefixed pattern of inhomogeneity. This can be stated precisely as the following:\medskip

 $(\ast\ast)$ for any $X\in [\omega_1]^{\omega_1}$, for any $k<\omega$, for any $\langle c_{ij}: i,j<k\rangle\in \omega^{k\times k}$ such that $c_{ii}=0$ for $i<k$, there is an uncountable
  family $\mathscr{A}\subset [X]^k$ of pairwise disjoint sets such that for any $a,b\in \mathscr{A}$, if $a<b$, then $osc(a(i),b(j))=osc(a(i),b(i))+c_{ij}$ for $i,j<k$.\medskip

  However, unlike Corollary \ref{cor1}, $(\ast\ast)$ is independent of ZFC. We first discuss the positive direction.    We will need  \cite[Lemma 8]{pw}   in the following proof.
    \begin{lem}[\cite{pw}]\label{lempw}
    Suppose that $\mathscr{A}\subset [\omega_1]^k$,
    $\mathscr{B}\subset [\omega_1]^l$, and $M$ is a countable elementary submodel of $H(\aleph_2)$ containing $\mathscr{A},\mathscr{B}$. Let $\delta=M\cap \omega_1$ and $R$ in $\{=,>\}$.  If there are $a_0\in \mathscr{A}\cap [\omega_1\setminus \delta]^k$, $b_0\in \mathscr{B}\cap [\omega_1\setminus \delta]^l$ such that for any $i<k,j<l$, $\varrho_{1a_0(i)}(\max L(\delta,b_0(j)))$ $R$ $\varrho_{1b_0(j)}(\max L(\delta,b_0(j)))$, then there are two uncountable sub-families of pairwise disjoint sets $\mathscr{A}'\subset\mathscr{A}, \mathscr{B}'\subset \mathscr{B}$ such that for any $a\in \mathscr{A}'$, $b\in \mathscr{B}'\cap [\omega_1\setminus \max a]^l$, for any $i<k,j<l$, $$osc(a(i),b(j))=osc(a(i),b(0))+c_{ij} \text{ where} $$ $c_{ij}=|Osc(\varrho_{1a_0(i)},\varrho_{1b_0(j)};L(\delta,b_0(j)))|-|Osc(\varrho_{1a_0(i)},\varrho_{1b_0(0)};L(\delta,b_0(0)))|$.

    Moreover, if we can choose $a_0=b_0$, then we can choose $\mathscr{A}'=\mathscr{B}'$. 
  \end{lem}

  \begin{thm}\label{**consistency}
    After adding a Cohen real, there exists a $C$-sequence so that $(\ast\ast)$ holds.
  \end{thm}
  \begin{proof}
  Let $\Lambda$ denote the set of non-zero countable limit ordinals. In the ground model, partition $\Lambda$ into countably many unbounded subsets $\langle \Lambda_n   \mid  n<\omega \rangle$. For $X\subset \omega_1$, let $acc(X)$ be the collection of accumulation points of $X$, i.e., $\alpha\in acc(X)$ iff $X\cap \alpha$ is unbounded below $\alpha$.
  
   Fix in the ground model a family of almost disjoint family $\{A_\alpha\subset \omega: \alpha\in \Lambda\}$ such that 
  \begin{enumerate}[$(i)$]
  \item $\min A_\alpha>n$ whenever $\alpha\in \Lambda_n$. 
  \end{enumerate}
  Then for each $\alpha\in \Lambda$, fix a bijection $\pi_\alpha: A_\alpha\rightarrow \alpha$ satisfying the following:
\begin{enumerate}[$(i)$]\setcounter{enumi}{1}
  \item for any $\beta\in \bigcap\limits_{n\in \omega} acc(\Lambda_n)\cap \alpha$, there are unbounded many $\gamma\in \Lambda\cap\beta$ such that $\pi_\alpha^{-1}(\gamma)< \min A_\gamma$.
    \end{enumerate}
    To see the existence of $\pi_\alpha$, let $\langle \beta_n: n<\omega\rangle$ list $\beta\in \bigcap_{n\in \omega} acc(\Lambda_n)\cap \alpha$ such that each appears $\omega$ many times and an enumeration $\langle\alpha_n: n<\omega\rangle$ of $\alpha$. We   define $\pi_\alpha(A_\alpha(n))$ injectively by induction on $n$. At step $2n$, arbitrarily define $\pi_\alpha(A_\alpha(2n))$ so that $\alpha_n$ is in the range. At step $2n+1$,  choose $\gamma\in \Lambda_{A_\alpha(2n+1)}\cap \beta_n$ greater than  $\{\pi_\alpha(A_\alpha(i)): i\leq 2n\}\cap \beta_n$ and define $\pi_\alpha(A_\alpha(2n+1))=\gamma$. Note that by $(i)$, $\pi_\alpha^{-1}(\gamma)=A_\alpha(2n+1)<\min A_\gamma$. It is  easy to check that $\pi_\alpha$ is as desired.\medskip

  Now we force with the Cohen forcing $\mc{P}$ where each condition $p$ is a   map from $m$ to $2$ for some $m<\omega$. The order is extension as functions. As usual, from a generic filter $G$, we get a Cohen real $r=\{n: p(n)=1$ for some $p\in G\}$.\medskip

  Before proceeding to the proof, we describe the $C$-sequence $\{C^r_\alpha: \alpha<\omega_1\}$ induced from $A_\alpha$'s, $\pi_\alpha$'s and the Cohen real $r$.\footnote{See \cite[Lemma 2.2.17]{st07} for a related but different construction.}

  For zero, set $C^r_{0}=\emptyset$. For successor case, set $C^r_{\alpha+1}=\{\alpha\}$. For non-zero limit case, we inductively define the $n$-th member of $C^r_\alpha$ for $\alpha\in \Lambda$:
\begin{enumerate}[$(i)$]\setcounter{enumi}{2}
  \item $C_\alpha^r(0)=\pi_\alpha(\min(A_\alpha\cap r))$;\\
$C_\alpha^r(n+1)=\pi_\alpha(\min(\{k\in A_\alpha\cap r: k>\pi_\alpha^{-1}(C_\alpha^r(n)) \text{ and } \pi_\alpha(k)>C_\alpha^r(n)\}))$.
 \end{enumerate}
 
   An easy density argument will show that the minimum in the definition of $C^r_\alpha(l)$ ($l=0$ or $n+1$) exists and hence $\{C^r_\alpha: \alpha<\omega_1\}$ is well-defined.\medskip

   Next we introduce \emph{the minimal walk decided by a condition $p$}. For a condition $p:m\rightarrow 2$ and any ordinal $\alpha< \omega_1$, let $C^p_\alpha$ be the finite set 
   $$\{\xi: p\Vdash \xi\in C^{\dot{r}}_\alpha\}.$$ 
   Note that   $C^p_\alpha$ can be equivalently defined in the following way:
\begin{enumerate}[$(i)$]\setcounter{enumi}{3}
\item $C^p_{\alpha+1}=\{\alpha\}$;
  \item for  non-zero countable limit $\alpha$,
  \begin{itemize}
  \item $C_\alpha^p(0)=\pi_\alpha(\min(A_\alpha\cap p^{-1}(\{1\})))$ if exists;
 \item $C_\alpha^p(n+1)=\pi_\alpha(\min(\{k\in A_\alpha\cap p^{-1}(\{1\}): k>\pi_\alpha^{-1}(C_\alpha^p(n)) \text{ and } \pi_\alpha(k)>C_\alpha^p(n)\}))$ if exists;
\item Otherwise if the minimum does not exist, then $C^p_\alpha(l)$ ($l=0$ or $n+1$) is undefined.
\end{itemize}
  \end{enumerate}
  Recall that the domain of each condition $p$ is a natural number.  So it is easy to deduce from definitions $(iii)-(v)$ that for any $\alpha<\omega_1$,
  \begin{enumerate}[$(i)$]\setcounter{enumi}{5}
 \item for any $p< q$, $C^q_\alpha$ is an initial segment of $C^p_\alpha$; for any $p\in G$, $C^p_\alpha$ is an initial segment of $C^r_\alpha$.
  \end{enumerate}
  
  For each $\alpha<\omega_1$, we also inductively define $F^p_\alpha\subset \alpha+1$ which codes all the walks from $\alpha$ decided by a condition $p$:
\begin{enumerate}[$(i)$]\setcounter{enumi}{6}
 \item $F_\alpha^p=\{\alpha\}\cup \underset{\xi\in C^p_\alpha}{\bigcup} F^p_\xi$.
 \end{enumerate}
  Note that for any $\xi\in F^p_\alpha$, there is a sequence of ordinals in $F^p_\alpha$, $\alpha=\alpha_0>\alpha_1>...>\alpha_n=\xi$ such that $\alpha_{i+1}\in C^p_{\alpha_i}$ for any $i<n$. Hence each $\xi\in F^p_\alpha$ associates a node in
  $$T=\{s\in (\alpha+1)^{<\omega}: s(0)=\alpha\text{ and }s(i+1)\in C^p_{s(i)}\text{ for any }i+1<|s| \}.$$

  Since $T$ is a finite branching tree with no infinite branch, $T$ and hence 
  \begin{enumerate}[$(i)'$]\setcounter{enumi}{6}
  \item $F^p_\alpha$ is finite.
 \end{enumerate}
  We now inductively define the upper trace and lower trace decided by  $p$.

 \begin{enumerate}[$(i)$]\setcounter{enumi}{7}
  \item  $Tr^p(\alpha,\beta)=\{\beta\}$ if $C^p_\beta\subset \alpha$;\\
  $Tr^p(\alpha,\beta)=\{\beta\}\cup Tr^p(\alpha, \min (C^p_\beta\setminus \alpha))$ otherwise.\\
 $L^p(\alpha,\beta)=\emptyset$ if $C^p_\beta=\emptyset$ or $\alpha=\beta$;\\
 $L^p(\alpha,\beta)=\{\max (C^p_\beta\cap \alpha)\}$ if $C^p_\beta\neq \emptyset$, $\beta>\alpha$ and $C^p_\beta\setminus \alpha=\emptyset$;\\
 $L^p(\alpha,\beta)=L^p(\alpha,\min (C^p_\beta\setminus \alpha))\cup\{\max (C^p_\beta\cap \alpha)\}\setminus \max (C^p_\beta\cap \alpha)$ otherwise.
 \end{enumerate}

    We also define the maximal weight decided by $p$:
 \begin{enumerate}[$(i)$]\setcounter{enumi}{8}
   \item $\varrho_1^p(\alpha,\beta)=\max\{|C^p_\xi\cap \alpha|: \xi\in Tr^p(\alpha,\beta)\setminus\{\alpha\}\}$.\medskip
\end{enumerate}

It is straightforward to check that for $\alpha\in Tr^p(\alpha,\beta)$,
\begin{enumerate}[(I)]
\item $p\Vdash ``Tr^p(\alpha,\beta)=Tr(\alpha,\beta),L^p(\alpha,\beta)=L(\alpha,\beta)$ and $\varrho^p_1(\alpha,\beta)=\varrho_1(\alpha,\beta)$''.
\end{enumerate}
Here $Tr(\cdot,\cdot), L(\cdot,\cdot), \varrho_1(\cdot,\cdot)$ are computed from the C-sequence $\{C^r_\alpha: \alpha<\omega_1\}$.


  Now we are ready to prove that the $osc$ defined from $\{C^r_\alpha: \alpha<\omega_1\}$ satisfies $(\ast\ast)$. Fix $X\in [\omega_1]^{\omega_1}$ and a continuous $\in$-chain $\{N_\alpha:\alpha<\omega_1\}$ of countable elementary submodels of $H(\aleph_2)$ containing everything relevant. Choose a club $C\subset \{N_\alpha\cap \omega_1:\alpha<\omega_1\}$. Going to subsets, we may assume that $X,C$ are in the ground model. By Lemma \ref{lempw}, it suffices to find $\delta \in C$ and
  $a\in [X\setminus \delta]^k$ such that 
  \begin{enumerate}[($\star$)]
  \item  for any $i, j<k$,  $|Osc(\varrho_{1a(i)},\varrho_{1a(j)};L(\delta,a(j)))|=c_{ij}$ and \hfill \break
  $\varrho_{1a(i)}(\max L(\delta,a(j)))=\varrho_{1a(j)}(\max L(\delta,a(j)))$.
  \end{enumerate}
  
  To see this, assume $(\star)$ is satisfied for $\delta\in C$ and $a\in [X\setminus \delta]^k$. By elementarity,
  $$A=\{\delta'<\omega_1: (\star) \text{ is satisfied when replacing } \delta, a \text{ by }\delta' \text{ and some } a'\in [X\setminus \delta']^k\}$$
   is stationary. For each $\delta'\in A$, fix a witness $a_{\delta'}\in [X\setminus \delta']^k$. Going to a stationary subset $A'\subset A$, we get $\mathscr{A}=\{a_{\delta'}: \delta'\in A'\}$ that is pairwise disjoint. Now applying Lemma \ref{lempw} to $\mathscr{A}$,  some countable elementary submodel $M$ such that $M\cap \omega_1\in A'$ and $a_{M\cap \omega_1}$, we get an uncountable $\mathscr{A}'\subset \mathscr{A}$. Then $\mathscr{A}'$ is a witness for $(**)$.\medskip

  Via the density argument, we only need to find, for each condition $p\in \mathcal{P}$,   
  a condition $q\leq p$, $\delta \in C$ and
  $a\in [X\setminus \delta]^k$ such that $q$ forces ``$a,\delta$ satisfy $(\star)$''.

  Now we fix a condition $p\in \mathcal{P}$. First find uncountable $X'\subset X$, $L\in [\omega_1]^{<\omega}$ and $s: L\rightarrow \omega$ such that
 \begin{enumerate}[(a)]
 \item $\{F^p_\alpha: \alpha\in X'\}$ forms a $\Delta$-system with root $F$;
 \item  for any $\alpha\in X'$,   $L^p(\max F+1,\alpha)=L$ and $\varrho^p_{1\alpha}\upharpoonright_L=s$.
 \end{enumerate}
This can be done since by $(vii)'$, $F^p_\alpha$ is finite.

Note  by $(vii), (viii)$, for any $\alpha\in X'$,
\begin{enumerate}[(a)]\setcounter{enumi}{2}
\item $\min Tr^p(\max F +1,\alpha)\in F^p_\alpha\setminus F$. So by (a), for any $\beta\neq \alpha$ in $X'$, $\min Tr^p(\max F +1,\beta) \neq \min Tr^p(\max F +1,\alpha)$.
\end{enumerate}
   Let $B=\{ \min Tr^p(\max F +1,\alpha): \alpha\in X'\}$.\medskip

  \noindent\textbf{Claim.} Suppose $p'\leq p$, $\delta\in C\setminus (\max F+1)$, $b\in [\bigcap\limits_{n<\omega}acc(\Lambda_n)\setminus \delta]^{4}$, $a\in [\omega_1]^k$ such that $\min Tr^{p'}(\delta, a(i))>b(3)$ for any $i<k$ and $\min Tr^{p'}(\delta, a(i))$'s are pairwise distinct. Suppose moreover that for any $i,j<k$,
  \begin{enumerate}
 \item   for any $\nu\in L^{p'}(\delta,a(j))$, $\nu\in Tr^{p'}(\nu,a(i))$;
   \item $\varrho^{p'}_{1a(i)}(\max L^{p'}(\delta,a(j)))=\varrho^{p'}_{1a(j)}(\max L^{p'}(\delta,a(j)))$.
   \end{enumerate}
   Then for any $i_0\neq j_0$ less than $k$, there is a condition $q\leq p'$ such that for any $i,j<k$, (1)-(2) hold when replacing $p'$ by $q$ and

  \begin{enumerate}\setcounter{enumi}{2}
    \item $L^{p'}(\delta,a(i))$ is a proper initial segment of $L^{q}(\delta,a(i)))$;

    \item $\min Tr^{q}(\delta, a(i))$'s are pairwise distinct and above $b(0)$;

    \item there is a $\xi>\max L^{p'}(\delta,a(j_0))$ such that 
    $$Osc(\varrho^{q}_{1a(i)},\varrho^{q}_{1a(j)};L^{q}(\delta,a(j)))=Osc(\varrho^{p'}_{1a(i)},\varrho^{p'}_{1a(j)};L^{p'}(\delta,a(j)))\cup \chi_{ij}$$ where $\chi_{ij}=\{\xi\}$ if $(i,j)=(i_0,j_0)$ and $\chi_{ij}=\emptyset$ otherwise;

  \end{enumerate}
  \begin{proof}[Proof of Claim.]
    Let $\beta'_i=\min Tr^{p'}(\delta,a(i))$ for any $i<k$. Note by our assumption, $\beta'_i$'s are pairwise distinct and above $b(3)$. Choose $n_0$ greater than 
    $$\max(\{\max \varrho^{p'}_{1a(i)}\upharpoonright_{L^{p'}(\delta,a(j))}: i,j<k\}\cup\{|C^{p'}_{\beta'_i}\cap \delta|: i<k\}\cup\{dom(p')\})$$
     such that 
    $$\{A_\alpha\setminus n_0: \alpha \in Tr^{p'}(\delta,a(i))\text{ for some }i<k\}\text{ are pairwise disjoint}.$$
    
    Firstly, we deal with requirement (3). By $(ii)$, find pairwise distinct elements $\langle \beta_i\in \Lambda\cap (b(2),b(3)): i<k\rangle$ such that   for $i<k$,
    $$ \pi^{-1}_{\beta'_i}(\beta_i)< \min A_{\beta_i}\text{ and}$$
    $$\pi^{-1}_{\beta'_i}(\beta_i)>\max(\{n_0\}\cup \{  \pi^{-1}_{\beta'_j}(\beta_j): j<i\}\cup \bigcup\{A_{\beta'_i}\cap A_{\beta_j}: j<i\}).$$
    Extend $p'$ to $q_0$ such that 
    $$q_0(\pi_{\beta'_i}^{-1}(\beta_i))=1\text{ for } i<k \text{ and}$$
     $$q_0(m')=0\text{ for the rest }m'\in dom(q_0)\setminus dom(p').$$
     
     We now check that for any $i<k$,
      \begin{enumerate}[(6.1)]
      \item $C^{q_0}_{\beta'_i}=C^{p'}_{\beta'_i}\cup \{\beta_i\}$;
      \item $C^{q_0}_{\beta_i}=\emptyset$.
       \end{enumerate}     
       To see (6.1), first recall that  $n_0 <\pi^{-1}_{\beta'_i}(\beta_i)$ and $(A_{\beta_i'}\setminus n_0)\cap A_{\beta_j'}=\emptyset$ for $i\neq j$.  Then $A_{\beta'_i}\cap q_0^{-1}\{1\}\setminus p'^{-1}\{1\}=\{\pi_{\beta'_i}^{-1}(\beta_i)\}$. So by $(vi)$, we only need to show that $\beta_i\in C^{q_0}_{\beta'_i}$. We will prove for the case $C^{p'}_{\beta'_i}\neq \emptyset$ and the other case is easier to handle.  On one hand, by $(v)$ and our choice of $n_0$, $\beta_i$, $\pi^{-1}_{\beta'_i}(\max(C^{p'}_{\beta'_i}))\in dom(p')<n_0<\pi^{-1}_{\beta'_i}(\beta_i)$. On the other hand, by our definition $\beta'_i=\min Tr^{p'}(\delta, a(i))$, $C^{p'}_{\beta_i'}\subset \delta< b(1)<\beta_i$. Now   (6.1) follows from $(v)$.
       
       (6.2) follows from the fact that for $i, j<k$, $dom(p')<n_0<\pi^{-1}_{\beta'_i}(\beta_i)<\min A_{\beta_i}$ and $\pi^{-1}_{\beta'_j}(\beta_j)\notin A_{\beta_i}$. To see this, for $j\leq i$, $\pi^{-1}_{\beta'_j}(\beta_j)\leq \pi^{-1}_{\beta'_i}(\beta_i)<\min A_{\beta_i}$. For $j>i$, $\max(A_{\beta'_j}\cap A_{\beta_i})< \pi^{-1}_{\beta'_j}(\beta_j)\in A_{\beta'_j}$.
       
       Then we check the following.
       \begin{enumerate}[(6.1)]\setcounter{enumi}{2}
      \item $Tr^{q_0}(\delta, a(i))=Tr^{p'}(\delta, a(i))\cup \{\beta_i\}$ and $L^{q_0}(\delta,a(i))=L^{p'}(\delta,a(i))$.
      \end{enumerate} 
      The first equality follows from (6.1), (6.2) and our definition $\beta'_i=\min Tr^{p'}(\delta, a(i))$. The second equality follows from (6.1), (6.2) and $(viii)$.
      
      Similarly, (6.1)-(6.3) and $(ix)$ show the following.
      \begin{enumerate}[(6.1)]\setcounter{enumi}{3}
      \item  $\varrho_{1a(i)}^{q_0}\up_{L^{q_0}(\delta,a(i))}=\varrho_{1a(i)}^{p'}\up_{L^{q_0}(\delta,a(i))}$.
       \end{enumerate} 
     \medskip
     
         Secondly, we find $\xi$ for requirement (5). Choose $n_1>dom(q_0)$ such that 
    $$ \{A_\alpha\setminus n_1: \alpha \in Tr^{q_0}(\delta,a(i))\text{ for some }i<k\} \text{ are pairwise disjoint}.$$
         Choose $\xi'_0<...<\xi'_{n_1}$ in $(\max\bigcup\{L^{q_0}(\delta,a(i)):i<k\},\delta)$ such that $$n_1<\pi_{\beta_{j_0}}^{-1}(\xi'_0)<...<\pi_{\beta_{j_0}}^{-1}(\xi'_{n_1}).$$
         We denote $\xi=\xi'_{n_1}$ and will show that $\xi$ satisfies requirement (5).

         By  $(ii)$, choose $\gamma'\in \Lambda\cap(b(1),b(2))$ such that $$\pi_{\beta_{j_0}}^{-1}(\xi'_{n_1})<\pi_{\beta_{j_0}}^{-1}(\gamma')< \min A_{\gamma'}.$$

    Extend $q_0$ to $q_1$ such that 
    $$q_1(\pi_{\beta_{j_0}}^{-1}(\alpha))=1\text{ for $\alpha\in\{\xi_j': j\leq n_1\}\cup \{\gamma'\}$ and }$$
    $$q_1(m)=0\text{ for any other }m\in dom(q_1)\setminus dom(q_0).$$ 
    
     Now it follows from (6.2) and our choice of $n_1$, $\xi'_j$'s,  $\gamma'$ and $q_1$ that
      \begin{enumerate}[(7.1)]
      \item $C^{q_1}_{\beta_i}=\emptyset$ for $i\in k\setminus \{j_0\}$, $C^{q_1}_{\beta_{j_0}}=\{\xi_j': j\leq n_1\}\cup \{\gamma'\}$ and $C^{q_1}_{\gamma'}=\emptyset$.
      \end{enumerate} 
      Together with (6.1) and (6.2), we get 
      \begin{enumerate}[(7.1)]\setcounter{enumi}{1}
     \item  for any $i\in k\setminus \{j_0\}$, $Tr^{q_1}(\delta, a(i))=Tr^{q_0}(\delta, a(i))$ and $L^{q_1}(\delta,a(i))=L^{q_0}(\delta,a(i))$;
     \item $Tr^{q_1}(\delta, a(j_0))=Tr^{q_0}(\delta, a(j_0))\cup \{\gamma'\}$ and $L^{q_1}(\delta,a(j_0))=L^{q_0}(\delta,a(j_0))\cup \{\xi \}$;

     \item $\xi \in Tr^{q_1}(\xi , a(j_0))$ and $\varrho_{1a(j_0)}^{q_1}(\xi )=n_1$.
       \end{enumerate} 
      \medskip
     
          Finally, we proceed toward the satisfaction of (1)-(5).  Choose $n_2>dom(q_1)$ such that 
    $$ \{A_\alpha\setminus n_2: \alpha \in Tr^{q_1}(\delta,a(i))\text{ for some }i<k\} \text{ are pairwise disjoint}.$$
          Choose $\tau_0<...<\tau_{n_2-1}$ in $(\xi+\omega,\delta)$, \\
     $\tau'_0<\tau'_1<...<\tau'_{n_1}<\xi<\tau'_{n_1+1}<\xi+\omega<\tau'_{n_1+2}<...<\tau'_{n_2-1}=\tau_{n_2-1}$ below $\delta$ and
     $\tau''_0<\tau''_1<...<\tau''_{n_1-1}<\xi<\tau''_{n_1}<\xi+\omega<\tau''_{n_1+1}<...<\tau''_{n_2-1}=\tau_{n_2-1}$ below $\delta$
     such that
     $$n_2<\pi_{\gamma'}^{-1}(\tau_0)<...<\pi_{\gamma'}^{-1}(\tau_{n_2-1}), \ n_2 <\pi_{\beta_{i_0}}^{-1}(\tau'_0)<...<\pi_{\beta_{i_0}}^{-1}(\tau'_{n_2-1})\text{ and}$$
      $$n_2<\pi_{\beta_{i}}^{-1}(\tau''_0)<...<\pi_{\beta_{i}}^{-1}(\tau''_{n_2-1})\text{ for }i\in k\setminus\{ i_0,j_0\}.$$

          For convenience, denote $\beta^*_i=\beta_i$ if $i\neq j_0$ and $\beta^*_{j_0}=\gamma'$.
          
          Then by $(ii)$, find pairwise distinct elements $\langle \gamma_i\in \Lambda\cap (b(0),b(1)): i<k\rangle$ such that 
          $$ \pi^{-1}_{\beta^*_i}(\gamma_i)< \min A_{\gamma_i}\text{ for $i< k$ and}$$
    $$\pi^{-1}_{\beta^*_i}(\gamma_i)>\max(\{\pi_{\beta^*_{j}}^{-1}(\tau_{n_2-1}): j< k\}\cup \{  \pi^{-1}_{\beta^*_j}(\gamma_j): j<i\}\cup \bigcup\{A_{\beta^*_i}\cap A_{\gamma_j}: j<i\}).$$

     Now extend $q_1$ to $q$ such that 
     $$q(\pi^{-1}_{\gamma'}(\gamma_{j_0}))=q(\pi^{-1}_{\gamma'}(\tau_m))=1\text{ for }m<n_2;$$ 
     $$q(\pi^{-1}_{\beta_{i_0}}(\gamma_{i_0}))=q(\pi^{-1}_{\beta_{i_0}}(\tau'_m))=1\text{ for }m<n_2;$$ 
      $$q(\pi^{-1}_{\beta_{i}}(\gamma_{i}))=q(\pi^{-1}_{\beta_{i}}(\tau''_m))=1\text{ for }m<n_2, i\in k\setminus\{i_0, j_0\}\text{ and}$$ 
     $$q(m)=0\text{ for the rest }m\in dom(q)\setminus dom(q_1).$$
     
         Now we check that
      \begin{enumerate}[(8.1)]
      \item $C^{q}_{\beta_i}=\{\tau_m'': m<n_2\}\cup \{\gamma_i\}$ for $i\in k\setminus \{i_0,j_0\}$, $C^{q}_{\beta_{i_0}}=\{\tau_m': m<n_2\}\cup \{\gamma_{i_0}\}$ and $C^{q}_{\gamma'}=\{\tau_m: m< n_2\}\cup \{\gamma_{j_0}\}$, $C^{q}_{\gamma_i}=\emptyset$ for $i<k$.
      \end{enumerate} 
      $C^{q}_{\beta_i^*}$ follows from our choice of $n_2$ and $q$. The argument for $C^{q}_{\gamma_i}=\emptyset$ is the same as the argument for (6.2).
      Then by (6.3), (7.2), (7.3), (8.1) and the fact that $\tau_{n_2-1}=\tau'_{n_2-1}=\tau''_{n_2-1}$, for any $i< k $, 
      \begin{enumerate}[(8.1)]\setcounter{enumi}{1}
     \item  $Tr^{q}(\delta, a(i))=Tr^{q_1}(\delta, a(i))\cup \{\gamma_i\}$ and $L^{q}(\delta,a(i))=L^{q_1}(\delta,a(i))\cup\{\tau_{n_2-1}\}$;
      
     \item $\tau_{n_2-1} \in Tr^{q}(\tau_{n_2-1} , a(i))$.
       \end{enumerate} 
       Also, by our choice of $\tau_m$'s, $\tau'_m$'s and $\tau''_m$'s,
         \begin{enumerate}[(8.1)]\setcounter{enumi}{3}
         \item $\varrho^{q}_{1a(i)}(\tau_{n_2-1})=n_2-1$ for $i<k$;
         \item $\varrho^{q}_{1a(i_0)}(\xi)=n_1+1$, $\varrho^{q}_{1a(i)}(\xi)=n_1$ for $i\neq i_0$.
         \end{enumerate} 
         To see, e.g., $\varrho^{q}_{1a(i_0)}(\xi)=n_1+1$, note $Tr^{q}(\xi, a(i_0))=Tr^{q_1}(\delta, a(i_0))\cup [\xi,\tau'_{n_1+1}]$. So $$\varrho^{q}_{1a(i_0)}(\xi)=\max\{|C^{q}_\eta\cap \xi |: \eta\in Tr^{q}(\xi,a(i_0))\setminus\{\xi\}\}=|C^{q}_{\beta_{i_0}}\cap \xi|=n_1+1.$$
      \medskip
     
       We first summarize several conditions from (6.1) to (8.5). By (6.3), (7.2) and (8.2), for $i\in k\setminus \{j_0\}$;
     \begin{enumerate}[(9.1)]
     \item $Tr^{q}(\delta, a(i))=Tr^{p'}(\delta, a(i))\cup \{\beta_i, \gamma_i\}$;
     \item $L^{q}(\delta, a(i))=L^{p'}(\delta, a(i))\cup \{\tau_{n_2-1}\}$.
     \end{enumerate} 
     By (6.3), (7.3) and (8.2),
     \begin{enumerate}[(9.1)]\setcounter{enumi}{2}
     \item $Tr^{q}(\delta, a(j_0))=Tr^{p'}(\delta, a(j_0))\cup \{\beta_{j_0}, \gamma', \gamma_{j_0}\}$;
    \item $L^{q}(\delta, a(j_0))=L^{p'}(\delta, a(j_0))\cup \{\xi, \tau_{n_2-1}\}$.
     \end{enumerate} 
     
    Then  we check that (1)-(5) are satisfied.  
     
     (1). Fix $\nu\in L^{q}(\delta, a(j))$. If $\nu\in L^{p'}(\delta, a(j))$, then by (I), $q$ forces that $Tr^q(\nu, a(i))=Tr(\nu, a(i))=Tr^{p'}(\nu, a(i))$. Hence $\nu\in Tr^q(\nu, a(i))$.
     
     If $\nu=\xi$, then $\nu\in Tr^q(\nu, a(j_0))$ by (7.4) and (I) and  $\nu\in Tr^q(\nu, a(i))$ for $i\in k\setminus \{j_0\}$,  by (9.1), (8.1)  and our choice of $\beta_i, \tau'_{n_1+1}, \tau''_{n_1}$. The case for $\nu=\tau_{n_2-1}$ is similar.
     
     (2) follows from (9.2), (9.4) and (8.4). 
     
     (3) follows from (9.2) and (9.4). 
     
     (4). Note by (8.2), for each $i<k$, $\gamma_i=\min Tr^q(\delta, a(i))$. Then (4)   follows from  our choice of $\gamma_i$. 
     
     (5). First note by (1) and (I), $\varrho^q_{1a(i)}$ agrees with $\varrho^{p'}_{1a(i)}$ on $L^{p'}(\delta, a(j))$. For $j\neq j_0$, (5) follows from (9.2) and (8.4). 
     
     For $j=j_0$ and $i\neq i_0$, (5) follows from (9.4), (7.4) and (8.5). Just note that by (I) and (1), $q$ forces $\varrho_{1a(j_0)}^{q}(\xi )=\varrho_{1a(j_0)} (\xi )=\varrho_{1a(j_0)}^{q_1}(\xi )$.
     
     For $(i,j)=(i_0, j_0)$, (5) follows from (2) for $p'$, (7.4), (8.5) and (9.4).     
     This finishes the proof of (5) and hence the claim.
  \end{proof}

Now fix $\delta\in C\setminus (\max F+1)$. Recall that by  (c), $\min Tr^p(\max F +1,\alpha)$'s are pairwise distinct. Choose $a\in [X']^k$ such that  for each $i<k$,
$$(\delta, \min Tr^{p}(\max F +1, a(i)))\cap \bigcap_{n<\omega} acc(\Lambda_n)\text{ is infinite}.$$
In particular,  for each $i<k$, $\min Tr^{p}(\max F +1, a(i))>\delta$ and hence $\min Tr^{p}(\max F +1, a(i))=\min Tr^{p}(\delta, a(i))$. Moreover, together with (b), $L^p(\delta, a(i))=L^p(\max F+1, a(i))=L$.
 Then (1)-(2) in the Claim are satisfied for $p, \delta$ and $ a$. To see e.g., (1), first note by $(viii)$,  $\nu\in Tr^{p}(\nu, a(i))$ for any  $\nu\in L^{p}(\delta, a(i))$. Then (1) follows from $L^p(\delta, a(i))=L^p(\delta, a(j))=L$.

Then use the Claim $\sum_{i,j<k} c_{ij}$ times we get a $q\leq p$ such that 
\begin{enumerate}\setcounter{enumi}{9}
\item $|Osc(\varrho^q_{1a(i)},\varrho^q_{1a(j)};L^q(\delta,a(j)))|=c_{ij}$;
\item $\varrho^{q}_{1a(i)}(\max L^{q}(\delta,a(j)))=\varrho^{q}_{1a(j)}(\max L^{q}(\delta,a(j)))$;
\item $C^q_{\min Tr^q(\delta,a(i))}=\emptyset$ (by (8.1)). 
\end{enumerate}

Now   pick some sufficiently large $n$ such that 
\begin{itemize}
\item $\pi^{-1}_{\min Tr^q(\delta,a(i))}(\delta+n)>dom (q)$;
\item  $\pi^{-1}_{\min Tr^q(\delta,a(i))}(\delta+n)\notin A_{\min Tr^q(\delta,a(j))}$ for $i\neq j$.
\end{itemize}
 Extend $q$ to $q'$ such that $q'(\pi_{\min Tr^q(\delta,a(i))}(\delta+n))=1$ for $i<k$ and $q'(m)=0$ for any other $m\in dom(q')\setminus dom(q)$.  It is straightforward to check that for any $i<k$,
 \begin{enumerate}\setcounter{enumi}{12}
 \item $C^{q'}_{\min Tr^q(\delta,a(i))}=C^{q}_{\min Tr^q(\delta,a(i))}\cup \{\delta+n\}$.
 \end{enumerate}
 Consequently, for any $i<k$,
  \begin{enumerate}\setcounter{enumi}{13}
 \item $Tr^{q'}(\delta,a(i))=Tr^q(\delta,a(i))\cup [\delta, \delta+n]$ and in particular $\delta\in Tr^{q'}(\delta,a(i))$;
 \item $L^{q'}(\delta,a(i))=L^q(\delta,a(i))$;
 \item $\varrho^{q'}_{1a(i)}\up_{L^{q'}(\delta,a(j))}=\varrho^{q}_{1a(i)}\up_{L^{q}(\delta,a(j))}$.
 \end{enumerate}

 By (I), this $q'$ will force ``$a,\delta$ satisfy $(\star)$''. This finishes the proof of the theorem.
 \end{proof}

Now we turn to the negative direction and show that it is consistent to have no $C$-sequence witnessing $(**)$.

\begin{thm}\label{t11}
  It is consistent that for any $C$-sequence, there is $X\in[\omega_1]^{\omega_1}$ such that for any uncountable family $\mathscr{A}\subset [X]^2$ of pairwise disjoint sets, there are $a<b$ in $\mathscr{A}$ and $i\neq j$ in $\{0,1\}$ such that $osc(a(i),b(j))\neq osc(a(i),b(i))+1$. In particular, $(**)$ fails.
\end{thm}

For any fixed $C$-sequence, we first define a forcing witnessing Theorem \ref{t11} for this sequence.

\begin{lem}\label{l15}
For any $C$-sequence, there is a proper forcing $\mathcal{P}$ of size $\omega_1$ such that $V^{\mathcal{P}}\models$ there is $X\in[\omega_1]^{\omega_1}$ such that for any uncountable family $\mathscr{A}\subset [X]^2$ of pairwise disjoint sets, there are $a<b$ in $\mathscr{A}$ and $i\neq j$ in $\{0,1\}$ such that $osc(a(i),b(j))\neq osc(a(i),b(i))+1$.
\end{lem}

Below we introduce some necessary lemmas. The following one follows directly from the proof of Lemma 8 in \cite{pw}.

\begin{lem}[\cite{pw}]\label{lem7}
  Assume that $\mathscr{A}\subset[\omega_1]^k$ is uncountable pairwise disjoint and $\langle c_{ij}: i,j<k\rangle\in \omega^{k\times k}$ such that for any $a,b\in \mathscr{A}$, if $a<b$, then $osc(a(i),b(j))=osc(a(i),b(i))+c_{ij}$ for $i,j<k$. Then for any stationary $S\subset\omega_1$, there are $\delta\in S$ and $a\in \mathscr{A}\cap [\omega_1\setminus \delta]^k$ such that $|Osc(\varrho_{1a(i)},\varrho_{1a(j)};L(\delta,a(j)))|=c_{ij}$ for $i,j<k$.
\end{lem}

Before proceeding to the proof, we first reduce the complexity of the conclusion of Lemma \ref{l15} for a given $C$-sequence.

\begin{lem}\label{lem30}
For a given  $X\in [\omega_1]^{\omega_1}$ and $\langle c_{ij}: i,j<k\rangle\in \omega^{k\times k}$ such that $c_{ii}=0$ whenever $i<k$, the following   statements are equivalent.
\begin{enumerate}
\item For any uncountable family $\mathscr{A}\subset [X]^k$ of pairwise disjoint sets, there are $a<b$ in $\mathscr{A}$ and $i, j<k$  such that 
$$osc(a(i),b(j))\neq osc(a(i),b(i))+c_{ij}.$$

\item There is a club $E$ such that for any $\delta\in E$, for any $a\in [X\setminus \delta]^k$, there are $i,j<k$ such that
$$|Osc(\varrho_{1a(i)},\varrho_{1a(j)}; L(\delta, a(j)))|\neq c_{ij}.$$

\item There is an uncountable set $A\subset \omega_1$ such that for any $\delta\in A$, for any $a\in [X\setminus \delta]^k$, there are $i,j<k$ such that
$$|Osc(\varrho_{1a(i)},\varrho_{1a(j)}; L(\delta, a(j)))|\neq c_{ij}.$$
\end{enumerate}
\end{lem}
\begin{proof}
$(1)\Rightarrow (2)$.  Let $\langle \mc{N}_\alpha: \alpha<\omega_1\rangle$ be a continuous $\in$-chain of countable elementary submodels of $H(\aleph_2)$ containing everything relevant. We will show that $E=\{\mc{N}_\alpha\cap \omega_1: \alpha<\omega_1\}$ is as desired.  Suppose towards a contradiction that for some $\delta\in E$ and $a\in [X\setminus \delta]^k$, 
\begin{enumerate}[(i)]
\item $|Osc(\varrho_{1a(i)},\varrho_{1a(j)}; L(\delta, a(j)))|= c_{ij}$  for any $i,j<k$.
\end{enumerate}

\noindent\textbf{Claim.}  There exists $\xi<\delta$ such that for any $\eta\in (\xi, \delta)$,   for any $i,j<k$,
\begin{enumerate}[(I)]
\item $|Osc(\varrho_{1a(i)},\varrho_{1a(j)}; L(\eta, a(j)))|= c_{ij}$;
\item $\varrho_{1a(i)}(\max L(\eta, a(j)))=\varrho_{1a(j)}(\max L(\eta, a(j)))$.
\end{enumerate}
\begin{proof}[Proof of Claim.]
By Fact \ref{f0}, choose $\tau<\delta$ such that for any $i,j<k$,
\begin{enumerate}[(i)]\setcounter{enumi}{1}
\item $\varrho_{1a(i)}\up_{[\tau, \delta)}=\varrho_{1a(j)}\up_{[\tau, \delta)}$;
\item $\tau> L(\delta, a(j))$.
\end{enumerate}

Let $\xi=\min (C_\delta\setminus \tau)$. Fix $\eta\in (\xi, \delta)$ and $i, j<k$. It suffices to prove that (I) and (II) hold for $\eta, i, j$.

Note that $\xi\leq \max (C_\delta\cap \eta)=\min L(\eta, \delta)$. So by (iii) and Fact \ref{f1},
\begin{enumerate}[(i)]\setcounter{enumi}{3}
\item $L(\delta, a(j))<\tau\leq \xi\leq L(\eta, \delta)$ and $L(\eta, a(j))=L(\eta, \delta)\cup L(\delta, a(j))$.
\end{enumerate}

Now (I) follows from (i), (ii) and (iv):
$$|Osc(\varrho_{1a(i)},\varrho_{1a(j)}; L(\eta, a(j)))|=|Osc(\varrho_{1a(i)},\varrho_{1a(j)}; L(\delta, a(j)))|= c_{ij}.$$
Also, (II) follows from (ii) and (iv).
\end{proof}

Assume $\delta=\mc{N}_\nu\cap \omega_1$ for some $\nu$. Then the following statement holds in $\mc{N}_\nu$ and hence in $H(\aleph_2)$:
\begin{enumerate}[(i)]\setcounter{enumi}{4}
\item for any  $\eta>\xi$,  there is $ a'\in [X\setminus \eta]^k $ such that (I), (II) in above claim hold with $ a $ replaced by $ a'.$
\end{enumerate}

Choose,  for each $\eta>\xi$,  $a_{\eta}\in [X\setminus \eta]^k$ witnessing (I) and (II). Then, for some stationary subset $\Gamma\subset (\xi, \omega_1)$, $\ms{B}=\{a_{\eta}: \eta\in \Gamma\}$ is pairwise disjoint. 

But applying Lemma \ref{lempw} to $\ms{B}$, some $\mc{M}\prec H(\aleph_2)$ containing everything relevant such that $\eta=\mc{M}\cap \omega_1\in \Gamma$ and $a_{\eta}$, we obtain an uncountable sub-family of $\ms{B}$ that is a counterexample to (1). This contradiction shows that (2) must be true.\medskip

$(2)\Rightarrow (3)$. Trivial.\medskip

$(3)\Rightarrow (1)$.   Let $\langle \mc{N}_\alpha: \alpha<\omega_1\rangle$ be a continuous $\in$-chain of countable elementary submodels of $H(\aleph_2)$ containing everything relevant and $E=\{\mc{N}_\alpha\cap \omega_1: \alpha<\omega_1\}$. We claim  that (2) holds. Just note that if $\delta\in E$ and $a\in [X\setminus \delta]^k$ is a counterexample to (2), then by the Claim, some $\eta\in A\cap \delta$ and $a$ would be a counterexample to (3). Now (1) follows from Lemma \ref{lem7} and (2).
\end{proof}

\textbf{Remark.} Fix a $C$-sequence. For $X\in [\omega_1]^{\omega_1}$ and $k<\omega$, let $\mc{C}_k(X)$ collect all $\langle c_{ij}: i,j<k\rangle\in \omega^{k\times k}$ with the property that
for all but countably many $\delta<\omega_1$, there exists $a\in [X\setminus \delta]^k$ such that 
$$|Osc(\varrho_{1a(i)},\varrho_{1a(j)}; L(\delta, a(j)))|=c_{ij}\text{ whenever } i, j<k.$$
Then as a consequence of Lemma \ref{lem30}, $\mc{C}_k(X)$ equals the collection of all pattern $\langle c_{ij}: i,j<k\rangle\in \omega^{k\times k}$ that can be realized inside $X$, in the sense that for some uncountable family $\ms{A}\subset [X]^k$ of pairwise disjoint sets, $osc(a(i), b(j))=osc(a(i), b(i))+c_{ij}$ whenever $a<b$ are in $\ms{A}$ and $i,j<k$. Moreover, if $\mc{M}\prec H(\aleph_2)$ is a countable elementary submodel containing the $C$-sequence and $X$,  then
$$\mc{C}_k(X)=\{\langle |Osc(\varrho_{1a(i)},\varrho_{1a(j)}; L(\mc{M}\cap \omega_1, a(j)))| : i ,j<k\rangle : a\in [X\setminus  \mc{M}]^k\}.$$
To see this, $\subset$ follows from countability of $\omega^{k\times k}$ and $\supset$ follows from Lemma \ref{lem30} (2) and elementarity.

\begin{lem}\label{lem28}
  Assume that $\mathcal {N}\prec H({2^\kappa}^+)$ is a countable elementary submodel for some uncountable cardinal $\kappa$, $\mathscr{A}\subset [\omega_1]^k$ is an uncountable family of pairwise disjoint sets in $\mathcal{N}$, $F\in [\omega_1\setminus \mathcal{N}]^{<\omega}$, $\tau\in \mathcal{N}\cap \omega_1$ and $R\in \{=,>\}$. There are countable $\mathcal{N}'\prec H(\kappa^+)$, $\mathscr{A}'\in [\mathscr{A}]^{\omega_1}$ and $\xi\in [\tau,\mathcal{N}'\cap \omega_1)$ such that $\mathscr{A}'\in \mathcal{N}'$, $\mathcal{N}'\in \mathcal{N}$ and for any $a\in \mathscr{A}'\cap \mathcal{N}'$, for any $i<k$, for any $\beta\in F$, $\xi\in L(a(i),\beta)$ and $\varrho_1(\xi,a(i))\ R\ \varrho_1(\xi,\beta)$.
\end{lem}
\begin{proof}
  First assume that $R$ is $>$. Let $E$ be the collection of all  $\mu<\omega_1$ such that
for any $\varepsilon<\omega_1$, $n<\omega$ and $\xi\in \omega_1\setminus \mu$, there is an $a\in \mathscr{A}\cap [\omega_1\setminus \varepsilon]^k$ with    $\varrho_1(\xi,a(i))>n$ whenever $i<k$.

  Recall that by proof of    \cite[Lemma 4.4]{moore06}, $E$ is uncountable. Note that $E\in \mathcal{N}$ and hence is unbounded in $\mathcal{N}$. Pick $\tau'\in E\cap \mc{N}\setminus \tau$ such that $\tau'> L(\mathcal{N}\cap \omega_1, \beta)$ for any $\beta\in F$.

  Now pick $\mathcal{N}'\in \mathcal{N}$ such that $\mathscr{A}\in \mathcal{N}'$, $\mathcal{N}'\prec H(\kappa^+)$ and $\xi=\max(\mathcal{N}'\cap C_{\mathcal{N}\cap \omega_1})>\tau'$. Let
   $$\mathscr{A}'=\{a\in \mathscr{A}: \xi<a\text{ and for any $i<k$,   }\varrho_1(\xi,a(i))> \max\{\varrho_1(\xi,\beta): \beta\in F\}\}.$$

   We claim that $\mathcal{N}',\mathscr{A}'$ and $\xi$ satisfy the requirement. To see this, fix $a\in \mathscr{A}'\cap \mathcal{N}'$,  $i<k$ and $\beta\in F$. Note that $\max(\mathcal{N}'\cap C_{\mathcal{N}\cap \omega_1})=\xi<a<\mc{N}'\cap \omega_1$. So $\xi=\min L(a(i), \mc{N}\cap \omega_1)$. Together with Fact \ref{f1} and the fact that $L(\mc{N}\cap\omega_1, \beta)<\tau'<\xi$, $ L(a(i), \beta)= L(a(i), \mc{N}\cap \omega_1)\cup  L(\mc{N}\cap \omega_1, \beta)$ and hence $\xi\in L(a(i), \beta)$. And by our definition of $\ms{A}'$, $\varrho_1(\xi,a(i))>\varrho_1(\xi,\beta)$.
   \medskip

  Then we assume that $R$ is $=$. Pick $a\in \mathscr{A}\setminus \mathcal{N}$. Pick $\tau''>\tau$ such that for any $\beta\in F$, for any $i<k$, $\tau''> L(\mathcal{N}\cap \omega_1, \beta)$ and $\varrho_{1\beta}\upharpoonright_{[\tau'',\mathcal{N}\cap \omega_1)}=\varrho_{1a(i)}\upharpoonright_{[\tau'',\mathcal{N}\cap \omega_1)}$. Now pick $\mathcal{N}'\in \mathcal{N}$ such that $\mathscr{A}\in \mathcal{N}'$, $\mathcal{N}'\prec H(\kappa^+)$ and $\xi=\max(\mathcal{N}'\cap C_{\mathcal{N}\cap \omega_1})>\tau''$. Let
  $$\mathscr{A}'=\{b\in \mathscr{A}: \xi<b\text{ and for any $\beta\in F$ and $i<k$, }\varrho_1(\xi,b(i))= \varrho_1(\xi,\beta)\}.$$

 Note $a\in \mathscr{A}'$ and $\mathscr{A}'\in \mathcal{N}'$. Hence $\mathscr{A}'$ is uncountable. Now an argument similar to previous case shows that $\mathcal{N}',\mathscr{A}'$ and $\xi$ are as desired.
\end{proof}

\begin{lem}\label{l14}
  Assume $\mathcal{N}\prec H(\beth_\omega)$,\footnote{Recall that $\beth_0=\aleph_0$, $\beth_{\alpha+1}=2^{\beth_\alpha}$ and $\beth_\alpha=\sup \{\beth_\beta: \beta<\alpha\}$ for limit $\alpha$.}  $F\in [\omega_1\setminus \mathcal{N}]^{<\omega}$, $n<\omega$, $\mathscr{A}\subset [\omega_1]^k$ is an uncountable family of pairwise disjoint sets in $\mathcal{N}$. Then there is an $a\in\mathscr{A}\cap\mathcal{N}$ such that for any $\beta\in F$ and $i,j<k$, $|Osc(\varrho_{1a(i)},\varrho_{1\beta};L(a(j),\beta))|>n$.
\end{lem}
\begin{proof}
   Applying previous lemma $2n+2$ times (choose $\tau$ large enough each time) we  get $\mathcal{N}_{2n+2}\in ...\in \mathcal{N}_1\in \mathcal{N}_0=\mathcal{N}$, $\mathscr{A}_{2n+2}\subset...\subset \mathscr{A}_1\subset\mathscr{A}_0=\mathscr{A}$ and $\xi_1<\xi_2<...<\xi_{2n+2}$ such that for any $m\in\{1,2,...,2n+2\}$, for any $a\in \mathscr{A}_{m}\cap \mathcal{N}_m$, $i<k$ and $\beta\in F$, $\xi_m\in L(a(i),\beta)$ and $\varrho_1(\xi_m,a(i))\ R\ \varrho_1(\xi_m,\beta)$ where $R$ is $=$ for $m$ odd and $R$ is $>$ for $m$ even.

  Now it is easy to see that for any $a\in \mathscr{A}_{2n+2}\cap \mathcal{N}_{2n+2}$, for any $i,j<k$, for any $\beta\in F$, $|Osc(\varrho_{1a(i)},\varrho_{1\beta};L(a(j),\beta))|>n$.
\end{proof}

\begin{proof}[Proof of Lemma \ref{l15}]
We first present the definition of the forcing.

\begin{defn}
  $\mathcal {P}$ consists of pairs $(c,F)\in[\omega_1]^{<\omega}\times[\omega_1]^{<\omega}$ such that for any $\eta\in c$, for any $\alpha,\beta\in F\setminus \eta$, either $|Osc(\varrho_{1\alpha},\varrho_{1\beta};L(\eta,\beta))|\neq 1$ or $|Osc(\varrho_{1\beta},\varrho_{1\alpha};L(\eta,\alpha))|\neq 1$. The forcing relation is defined by $(c,F)\leq (c',F')$ if $c\supset c'$ and $F\supset F'$.
\end{defn}

It is clear that the cardinality of $\mathcal{P}$ is $\omega_1$.\medskip

\noindent\textbf{Claim 1.}  $\mathcal{P}$ is proper.

\begin{proof}[Proof of Claim 1.]
  Fix a countable $\mathcal{M}\prec H(\kappa)$ containing $\mathcal{P}$ and $(c,F)\in \mathcal{M}\cap \mathcal{P}$. Let $\delta=\mathcal{M}\cap \omega_1$. It suffices to show that $(c\cup\{\delta\},F)$ is $(\mathcal{M},\mathcal{P})$-generic. Fix $D\in \mathcal{M}$ dense open and $(c',F')\in D$ stronger than $(c\cup\{\delta\},F)$. Fix partitions $c'=c_1\cup c_2$ and $F'=F_1\cup F_2$ where $c_1=c'\cap \mathcal{M}$ and $F_1=F'\cap \mathcal{M}$. Pick $\tau\in C_\delta$ such that for any $\beta\in F_2$, $\tau> L(\delta, \beta)\cup c_1\cup F_1$ and $\varrho_{1\beta}\upharpoonright_{[\tau,\delta)}=\varrho_{1\delta}\upharpoonright_{[\tau,\delta)}$.

  Let $\mathscr{A}$ collect all $c''\cup F''>\tau$ such that 
  \begin{enumerate}
  \item the order type of $c''$ (resp. $F''$) in $c''\cup F''$ is the same as the  order type of $c_2$ (resp. $F_2$) in $c_2\cup F_2$, i.e., $(c'', F'', <)$ is isomorphic to $(c_2, F_2, <)$;
  \item  $(c_1\cup c'',F_1\cup F'')\in D$;
  
  \item for any $\eta\in c_1$, for any $i,j<|F_2|$, $L(\eta,F''(i))=L(\eta,F_2(i))$ and \hfill \break 
  $\varrho_{1F''(i)}\upharpoonright_{L(\eta,F''(j))}=\varrho_{1F_2(i)}\upharpoonright_{L(\eta,F''(j))}$.
   \end{enumerate}

  Although $\mathscr{A}$ itself is not pairwise disjoint, it contains an uncountable pairwise disjoint subfamily. By Lemma \ref{l14}, we can find $c''\cup F''\in \mathscr{A}\cap \mathcal{M}$ such that
  \begin{enumerate}\setcounter{enumi}{3}
  \item for any $\eta\in c'', \alpha \in F'', \beta\in F_2$, $|Osc(\varrho_{1\alpha},\varrho_{1\beta};L(\eta,\beta))|>1$.
  \end{enumerate}

 Note by definition of $\mathscr{A}$,  $(c_1\cup c'',F_1\cup F'')\in D$.

  Now we just need to check that $(c_1\cup c''\cup c_2, F_1\cup F''\cup F_2)$ is a condition. Fix $\eta\in c_1\cup c''\cup c_2$ and $\alpha<\beta$ in $(F_1\cup F''\cup F_2)\setminus \eta$. It suffices to show that either $|Osc(\varrho_{1\alpha},\varrho_{1\beta};L(\eta,\beta))|\neq 1$ or $|Osc(\varrho_{1\beta},\varrho_{1\alpha};L(\eta,\alpha))|\neq 1$.\medskip

  \textbf{Case 1}: $\eta\in c_2$.\medskip

  Then $\alpha,\beta\in F_2$. So the satisfaction is inherited from condition $(c', F')$.\medskip

  \textbf{Case 2}: $\eta\in c_1$.\medskip

  \textbf{Subcase 2.1}: $\alpha,\beta\in F_1\cup F_2$ or $\alpha,\beta\in F_1\cup F''$.\medskip

  Trivial.\medskip

  \textbf{Subcase 2.2}: $\alpha\in F''$ and $\beta\in F_2$.\medskip

 Let $\alpha'\in F_2$ be in the same position as $\alpha$ in $F''$. 
 First assume that $\alpha'=\beta$. By $(3)$, $|Osc(\varrho_{1\alpha},\varrho_{1\beta};L(\eta,\beta))|=0$. 
 
 Then assume that $\alpha'\neq \beta$. By (3), $|Osc(\varrho_{1\alpha},\varrho_{1\beta};L(\eta,\beta))| $ $=|Osc(\varrho_{1\alpha'},\varrho_{1\beta};L(\eta,\beta))|$ and $|Osc(\varrho_{1\beta},\varrho_{1\alpha};L(\eta,\alpha))|=|Osc(\varrho_{1\beta},\varrho_{1\alpha'};L(\eta,\alpha'))|$. By forcing condition $(c', F')$ again, one of them does not equal 1.
 \medskip

  \textbf{Case 3}: $\eta\in c''$.\medskip

  \textbf{Subcase 3.1}: $\alpha,\beta\in F''$.\medskip

  Trivial.\medskip

  \textbf{Subcase 3.2}: $\alpha\in F''$ and $\beta\in F_2$.\medskip

 By (4), $|Osc(\varrho_{1\alpha},\varrho_{1\beta};L(\eta,\beta))|>1$.\medskip

  \textbf{Subcase 3.3}: $\alpha,\beta\in F_2$.\medskip
  
  Recall that $L(\delta,\beta)<\tau<\eta$ and $\tau\in C_\delta$. So $L(\delta,\beta)<\tau\leq L(\eta,\delta)$. By Fact \ref{f1},  $L(\eta,\beta)= L(\eta,\delta)\cup L(\delta,\beta)$.  
Recall also that $\varrho_{1\alpha}\upharpoonright_{[\tau,\delta)}=\varrho_{1\beta}\upharpoonright_{[\tau,\delta)}$.   
So 
$Osc(\varrho_{1\alpha},\varrho_{1\beta};L(\eta,\beta))= Osc(\varrho_{1\alpha},\varrho_{1\beta};L(\delta,\beta))$ and $Osc(\varrho_{1\beta},\varrho_{1\alpha};L(\eta,\alpha)) = Osc(\varrho_{1\beta},\varrho_{1\alpha};L(\delta,\alpha)).$ Since $\delta\in c'$ and $\alpha, \beta\in F'$, one of them does not equal 1.
\end{proof}

Now suppose that $G$ is  a generic filter for $\mathcal{P}$. Define
$$C=\bigcup\{c: (c,F)\in G\text{ for some }F\},$$
$$X=\bigcup\{F: (c,F)\in G\text{ for some }c\}.$$
The forcing condition guarantees the following.
\begin{enumerate}\setcounter{enumi}{4}
\item For any $\delta\in C$, for any $\alpha,\beta\in X\setminus \delta$, either $|Osc(\varrho_{1\alpha},\varrho_{1\beta};L(\delta,\beta))|\neq 1$ or $|Osc(\varrho_{1\beta},\varrho_{1\alpha};L(\delta,\alpha))|\neq 1$. 
\end{enumerate}

We may assume that $X$ is uncountable. To see this, fix countable $\mathcal{M}\prec H(\kappa)$ containing everything relevant and $\alpha\in \omega_1\setminus \mc{M}$. By the proof of Claim 1, $(\{\mathcal{M}\cap \omega_1\},\{\alpha\})$ is $(\mathcal{M},\mathcal{P})$-generic. We claim that  $(\{\mathcal{M}\cap \omega_1\},\{\alpha\})\Vdash \dot{X}$ is uncountable. Suppose otherwise, some $r< (\{\mathcal{M}\cap \omega_1\},\{\alpha\})$ forces $\sup \dot{X}<\omega_1$. Then by $(\mathcal{M},\mathcal{P})$-genericity, there are $r'<r$ and $\beta\in \mc{M}\cap \omega_1$ such that $r'\Vdash \sup \dot{X}=\beta$ (see, e.g., \cite[Lemma 31.6]{jech03}). But this contradicts the fact that $\alpha\geq \mc{M}\cap \omega_1>\beta$ and $(\{\mathcal{M}\cap \omega_1\},\{\alpha\})$ and hence $r'$ forces $\alpha\in \dot{X}$.
Now choose $G$ containing $(\{\mathcal{M}\cap \omega_1\},\{\alpha\})$.\medskip



Repeating above argument if necessary, we may assume that $C$ is uncountable. Then Lemma \ref{l15} follows from (5) and Lemma \ref{lem30}.
\end{proof}

\begin{proof}[Proof of Theorem \ref{t11}]
Start from a model of GCH. Define an $\omega_2$-step, countable support iterated forcing to deal with all $C$-sequences via a book-keeping way. For each single set and any pre-assigned $C$-sequence, we force with the forcing from Lemma \ref{l15}.  By Lemma \ref{lem30}, the conclusion of Lemma \ref{l15} is upward absolute between models with the same $\omega_1$.  A standard argument shows that the forcing is proper and $\omega_2$-cc and thus preserves cardinals. Thus in the final model, all $C$-sequences satisfy the conclusion of Theorem \ref{t11}.
\end{proof}

 \section*{Acknowledgement}
 We thank the referees for their careful reading of the manuscript and suggestions that improve the exposition.

\bibliographystyle{plain}

\begin{thebibliography}{10}

\bibitem{jwsc} J. W. S. Cassel, An introduction to Diophantine approximation, Cambridge Tracts in Mathematics and Mathematical Physics, No. 45.
Cambridge University Press, New York, 1957.




\bibitem{jech03}
T.~J. Jech, {\it Set theory}, the third millennium edition, revised and expanded., Springer Monographs in Mathematics, Springer, Berlin, 2003. 

\bibitem{lk} L. Kronecker, N\:aherungsweise ganzzahlige Aufl\:osung linearer Gleichungen, Monatsberichte der k\:oniglich Preussischen Akademie der Wissenschaften zu
Berlin vom Jahre 1884. S. 1179-1193, 1271-1299.

\bibitem{ku} K. Kuratowski and S. Ulam, Quelques propri\'et\'es topologiques du produit combinatoire, Fund. Math. 19 (1932), 247--251.

    

\bibitem{moore06} J. T. Moore, A solution to the L space problem, J. Amer. Math. Soc., 19 (3) (2006) 717-736.

\bibitem{moore08} J. T. Moore, An L space with a d-separable
square. Topology  Appl., 155 (2008) 304-307.



\bibitem{p} Y. Peng, An L space with non-Lindel\"of square, Topology Proc., 46 (2015), 233--242.

\bibitem{pw} Y. Peng and L. Wu,  A Lindel\"of group with non-Lindel\"of square, Adv. Math., 325 (2018), 215-242.


\bibitem{shak88} D. B. Shakhmatov, The structure of topological fields and cardinal invariants, Trans. Moscow Math. Soc., (1988) 251-261.

\bibitem{shak} D. B. Shakhmatov, A comparative survey of selected results and open problems concerning topological groups, fields and vector spaces, Topology  Appl., 91 (1999), 51-63.


\bibitem{st87} S. Todorcevic, Partitioning pairs of countable ordinals,  Acta. Math., 159 (1987), no. 3-4, 261-294.


\bibitem{st07} S. Todorcevic, Walks on Ordinals and their Characteristics,
Progress in Mathematics, 263. Birkh\"auser Verlag, Basel, 2007.
vi+324 pp. ISBN: 978-3-7643-8528-6.


\end{thebibliography}

\end{document}